\documentclass[aos, preprint]{imsart}
\setattribute{journal}{name}{} %suppress text Submitted to...

\RequirePackage[OT1]{fontenc}
\RequirePackage{amsthm,amsmath}
\RequirePackage[numbers]{natbib}
\RequirePackage[colorlinks,citecolor=blue,urlcolor=blue]{hyperref}

\usepackage{amssymb,amsmath,amsthm}
\usepackage{graphicx}

\usepackage{afterpage}% http://ctan.org/pkg/afterpage
\usepackage{float}% http://ctan.org/pkg/float
\usepackage{lipsum}% http://ctan.org/pkg/lipsum

\usepackage{amsthm}
\usepackage{subfig}
\usepackage{booktabs}

\newtheorem{theorem}{Theorem}[section]
\newtheorem*{theorem*}{Theorem}
\newtheorem{lemma}{Lemma}
\newtheorem*{lemma*}{Lemma}
\newtheorem{corollary}{Corollary}
\newtheorem*{corollary*}{Corollary}

\newtheorem*{remark*}{Remark}

\begin{document}   

\begin{frontmatter}
%\title{Title \thanksref{T1}}
\title{Distributions and Statistical Power of Optimal Signal-Detection Methods In Finite Cases}
\runtitle{Distributions of Signal-Detection Methods In Finite Cases}
%\thankstext{T1}{Footnote to the title with the ``thankstext'' command.}

\begin{aug}
\author{\fnms{Hong} \snm{Zhang}\thanksref{t1,m1}\ead[label=e2]{hzhang@wpi.edu}},
\author{\fnms{Jiashun} \snm{Jin}\thanksref{m2}\ead[label=e3]{jiashun@stat.cmu.edu}}
\and
\author{\fnms{Zheyang} \snm{Wu}\thanksref{t1,t3, m1}
\ead[label=e1]{zheyangwu@wpi.edu}
\ead[label=u1,url]{https://www.wpi.edu/people/faculty/zheyangwu}}

\thankstext{t1}{Partially supported by NSF grant DMS-1309960. }
%\thankstext{t2}{Partially supported by ... }
\thankstext{t3}{Corresponding author.}
\runauthor{H. Zhang et al.}

\affiliation{Worcester Polytechnic Institute\thanksmark{m1} and Carnegie Mellon University\thanksmark{m2}}

\address{Department of Mathematical Sciences and\\
Program of Bioinformatics and Computational Biology and\\
Program of Data Science\\
100 Institute Road\\
Worcester, MA 01609\\
\printead{e1}\\
\printead{u1}
}

\end{aug}

\begin{abstract}
In big data analysis for detecting rare and weak signals among $n$ features, some grouping-test methods such as Higher Criticism test (HC), Berk-Jones test (B-J), and $\phi$-divergence test share the similar asymptotical optimality when $n \rightarrow \infty$.  However, in practical data analysis $n$ is frequently small and moderately large at most. In order to properly apply these optimal tests and wisely choose them for practical studies, it is important to know how to get the p-values and statistical power of them. To address this problem in an even broader context, this paper provides analytical solutions for a general family of goodness-of-fit (GOF) tests, which covers these optimal tests. For any given i.i.d. and continuous distributions of the input test statistics of the $n$ features, both p-value and statistical power of such a GOF test can be calculated. By calculation we compared the finite-sample performances of asymptotically optimal tests under the normal mixture alternative. Results show that HC is the best choice when signals are rare, while B-J is more robust over various signal patterns. In the application to a real genome-wide association study, results illustrate that the p-value calculation works well, and the optimal tests have potentials for detecting novel disease genes with weak genetic effects. The calculations have been implemented in an R package {\it SetTest} and published on the CRAN. 
\end{abstract}

%\begin{keyword}[class=MSC]
%\kwd[Primary ]{60K35}
%\kwd{60K35}
%\kwd[; secondary ]{60K35}
%\end{keyword}
%
%\begin{keyword}
%\kwd{sample}
%\kwd{\LaTeXe}
%\end{keyword}

\end{frontmatter}

\section{Introduction}

In big data analysis, signals are often buried within a large amount of noises and are thus relatively weak and rare. It is ideal to apply the \emph{optimal tests} that are capable of detecting the minimal signals required by statistics. Through asymptotics, many theoretical studies have brought exciting results on designing such optimal tests. In particular, under the Asymptotically Rare and Weak (ARW) setting, the Higher Criticism (HC) and its various modifications, the Berk-Jones  (B-J) type tests, a spectrum of $\phi$-divergence type tests, etc. were studied and proven asymptotically optimal \cite{Donoho2004, Donoho2008, donoho2014higher, jager2007goodness}. These optimal tests are attractive in many scientific researches. For example, in large-scale genetic association studies, a main strategy to find disease-associated genes is to determine whether some of the genetic variants within candidate genes could affect disease outcome. Such genetic effects are often weak and rare, especially relative to the cumulated noise level in big data. 

However, for practical applications under finite cases the questions remain on 1) how to analytically calculate p-values as well as statistical power, and 2) what are the real performances of those methods that are asymptotically equivalent.  First, to obtain the p-value for error control, the Monte-Carlo simulation or permutation tests have significant limitations: (A) they require daunting computation, and (B) empirical p-values are discrete, causing ties among candidates. These issues are especially concerned when very small p-values are demanded to handle a huge number of simultaneous tests. Secondly, for those optimal tests with the same asymptotic property under $n \rightarrow \infty$, it is important to understand their relative performance for finite $n$. In order to solve these problems we need to calculate the distributions of these tests under both $H_0$ and $H_1$ at each given $n$. Comparing with the literature, this paper gives a complete answer by providing a comprehensive calculation for a broad family of relevant tests.  

\subsection{Limitations of current methods}

In general there are two types of methods for distribution calculation. The first is to calculate the exact distribution. Recursive methods (e.g., Noe's recursion\cite{noe1972calculation, noe1968calculation}, Bolshev's recursion \cite{kotel1983computing, shorack2009empirical}, Steck's recursion \cite{steck1969smirnov, breth1976recurrence}, Ruben's recursion \cite{ruben1976evaluation}, etc.) were developed to calculate the distribution of Kolmogorov-Smirnov type statistics under $H_0$. In a similar fashion, Barnett and Lin \cite{barnett2014analytical} provided a calculation method specifically for HC. Such recursive methods have heavy computation load, with complexity of $O(n^3)$. Moscovich, Nadler and Spiegelman \cite{moscovich2016exact} reduced the computation to $O(n^2)$. However, all these methods assume the full domain $\mathcal{R} = \{ 1 \leq i \leq n\}$ in the supremum formula of the relevant tests (see the HC statistic in (\ref{equ: HC}) as an example). However, calculation that allows arbitrary  $\mathcal{R}$ is important, because many tests improve performance by properly restricting $\mathcal{R}$. For example, a modified version of HC is under $\mathcal{R} = \{ 1 < i \leq n/2, p_{(i)} \geq 1/n \}$. This is because that including too small p-values $p_{(i)}$ could make HC heavy tailed; while including too big p-values may unnecessarily reduces computational efficiency (c.f. \cite{Donoho2004}). Moreover, these methods did not give statistical power calculation under $H_1$ yet. 

%\afterpage{
 \begin{figure}
	\caption{Comparison among different methods for calculating the right-side probability of the modified HC test (MHC) in (\ref{equ: HC}) with  $\mathcal{R} = \{ 1 < i \leq n/2, p_{(i)} \geq 1/n \}$\cite{Donoho2004} under $H_0: X_i \overset{\text{i.i.d.}}{\sim} N(0, 1)$. Black solid curves: by simulation; red dashed curves: by Corollary \ref{Corol: exactmodified}; green dotted curves: by Li and Siegmund's method\cite{Li2014higher}\label{fig:comparison0}. }
	\begin{center}
		\subfloat{\includegraphics[width=0.5\textwidth,height=1.6in]{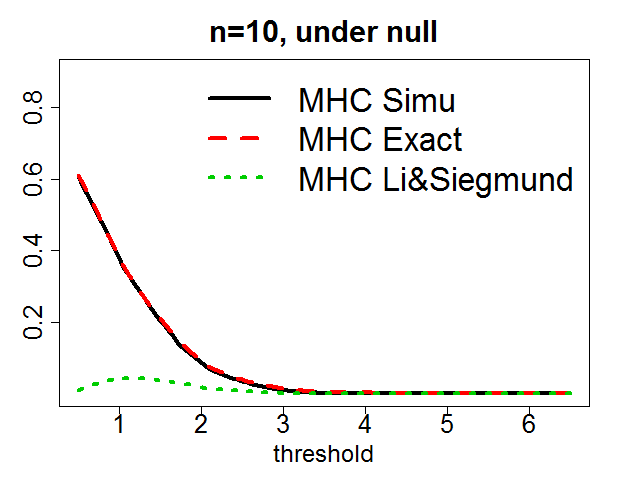}}
		\subfloat{\includegraphics[width=0.5\textwidth,height=1.6in]{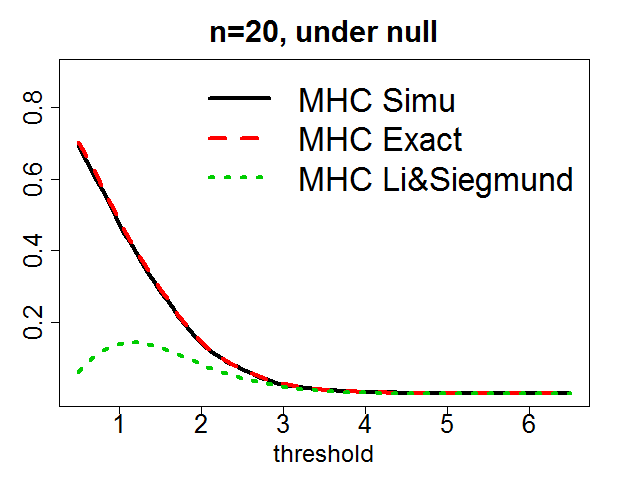}}\\
		\subfloat{\includegraphics[width=0.5\textwidth,height=1.6in]{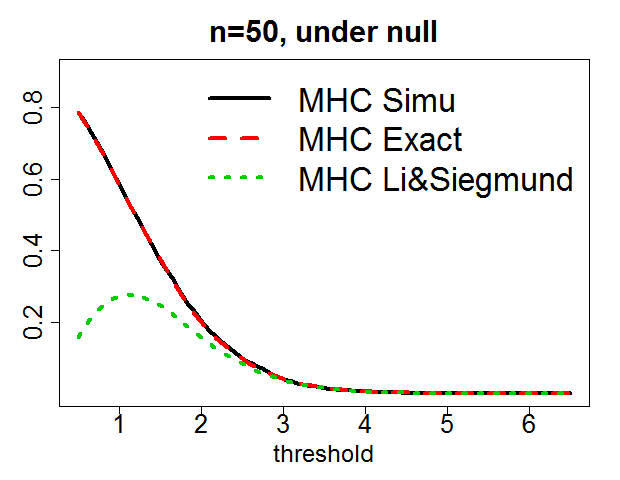}}
		\subfloat{\includegraphics[width=0.5\textwidth,height=1.6in]{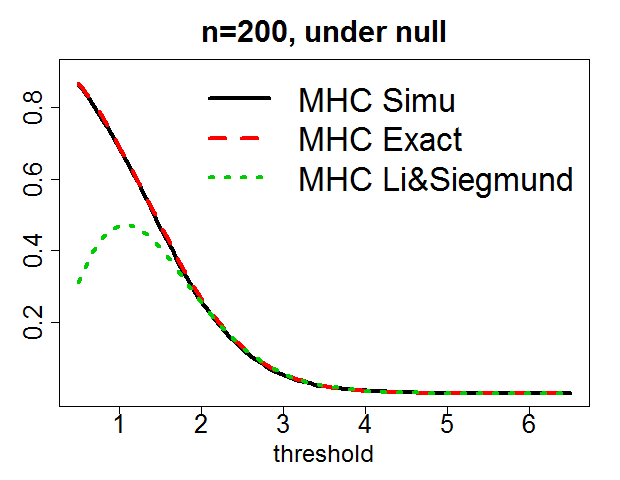}}\\
	\end{center}
 \end{figure}
% }
%\lipsum[4-6]

The second type methods are to approximate the distribution. Kolmogorov-Smirnov type statistics have been proven to converge in law to an extreme-value distribution \cite{eicker1979asymptotic,jaeschke1979asymptotic}. However, such convergence is too slow to be accurate for even moderately large $n$ \cite{barnett2014analytical}. Recently, Li and Siegmund (LS) \cite{Li2014higher} developed an asymptotic approximation method. One problems is that this method cannot approximate the whole distribution. Figure \ref{fig:comparison0} shows that the LS calculation (green dotted curves) fails to match the simulation (black solid curves) at small thresholds, whereas our calculation (e.g., the red dashed curves) reveal the whole distributions. Moreover, comparing with our study setting, LS assumes slightly more restriction for the supremum searching domain $\mathcal{R}$, and covers a narrower range of test types. Again, literature studies did not yet provide satisfiable method for deriving the relevant distributions under $H_1$ for statistical power calculation.

\subsection{Our contribution}

This paper has two folds of contributions. First, it provides techniques for calculating the distributions of a general family of goodness-of-fit (GOF) tests, which covers the optimal tests described above. We give calculation methods for the exact as well as the approximated distributions, balancing between accuracy and computation burden. The methods allow (A) arbitrary truncation strategies (e.g., $\mathcal{R}$ for HC in (\ref{equ: HC})), and (B) arbitrary null and alternative hypotheses as long as they are i.i.d. and continuous. With such techniques, p-values and statistical power can be calculated at any testing threshold. 

Second, based on analytical calculation we systematically compared the power of the asymptotically optimal tests. Simulations were applied for examing these tests under various $n$ values and signals patterns in practice. We also demonstrated the application of these weak-signal-sensitive tests in a real genome-wide association study (GWAS) for detecting genes associated to the Crohn's disease (see Figure \ref{fig:GWAS_Hexact}). 

\begin{figure}
	\caption{The association p-values for genes by exact calculation of four tests. First row: $HC^{2004}$ and Berk-Jones; second row: reverse Berk-Jones and $HC^{2008}$.}\label{fig:GWAS_Hexact}
	\begin{center}
		\subfloat{\includegraphics[width=0.5\textwidth,height=2in]{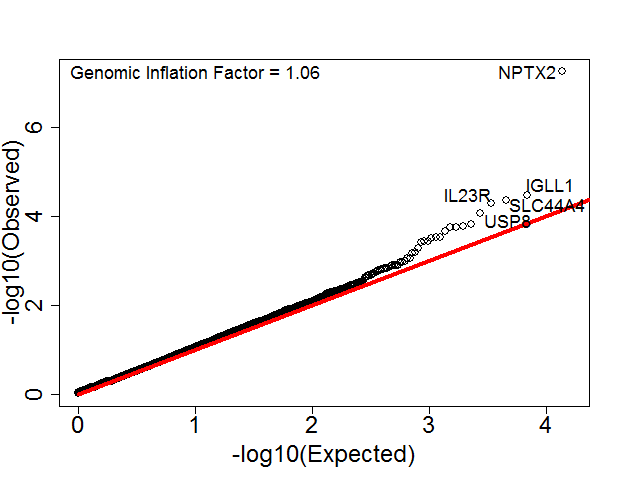}}
		\subfloat{\includegraphics[width=0.5\textwidth,height=2in]{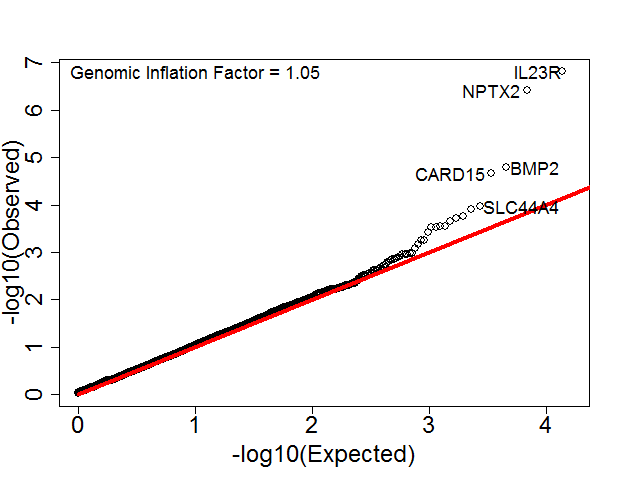}}\\
		\subfloat{\includegraphics[width=0.5\textwidth,height=2in]{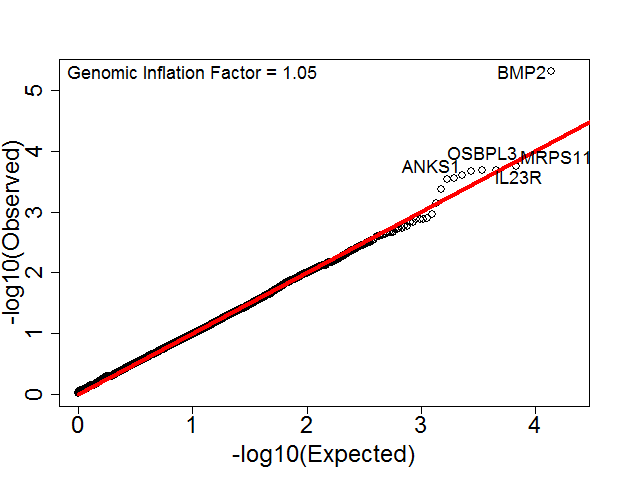}}
		\subfloat{\includegraphics[width=0.5\textwidth,height=2in]{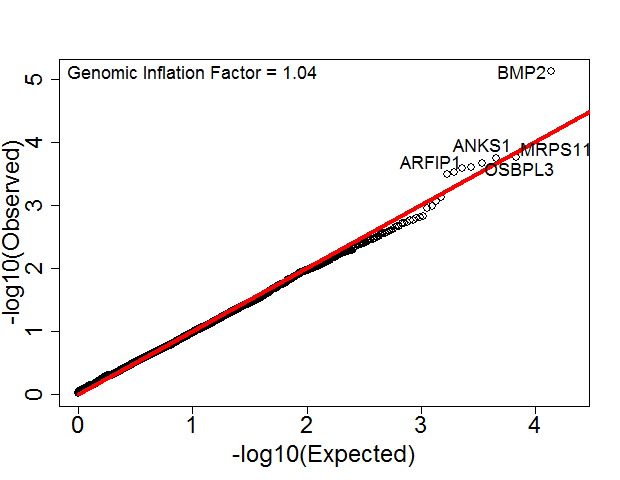}}
	\end{center}
\end{figure}

In summary, the innovative contribution of this work lies on providing a more mature statistical framework that revolves the problem of calculating distributions for a broader family of GOF tests under both $H_0$ and $H_1$. Thus, those asymptotically equivalent tests for detecting weak signals can be apply, as well as compared by power, in practical applications. The paper is organized as follows. In section \ref{Sect: Formulation} we formulate the problem under a general setup of hypotheses and GOF tests, and brief the essence of our strategy. The analytical results are presented in Section \ref{Sect: AnaResults} for both exact and approximated calculations. Through simulations Section \ref{Sect: NumResults} numerically evidences the calculation accuracy, and provides power comparisons among the asymptotically optimal tests. We show the application of the GOF tests in a real GWAS in Section \ref{Sect: GWAS}. In Section \ref{Sect: Discu} we discuss relevant theoretical and practical issues. The techniques of relevant proofs and lemmas are given in the Appendix.

\section{Problem Formulation and General Strategy\label{Sect: Formulation}}

\subsection{Background on Detection of Weak and Sparse Signals}

We consider the \emph{signal detection} problem through testing whether there exist ``signals" in noisy background. Under the broadly applicable Gaussian Means Model, signals refer to the nonzero means in the normally distributed data values (or statistics) $X_1,X_2,...,X_n$. A test is said of detecting signals if it successfully concludes the alternative when it is true. A typical setting for the null and the alternative is
\begin{equation}
	H_0: X_i \overset{\text{i.i.d.}}{\sim}  N(0, 1), \quad
	H_1: X_i \overset{\text{i.i.d.}}{\sim} \epsilon N(\mu, 1)  + (1-\epsilon)N(0, 1),
	\label{equ: mixNormalHypo}
\end{equation}
where $\epsilon \in (0, 1)$ denotes the expected proportion of signals. 

Under the Asymptotic Rare and Weak (ARW) setting, the proportion of signals is $\epsilon_n=n^{-\alpha}$, $\alpha \in (1/2, 1)$, and the mean is $\mu_n=\sqrt{2r\log(n)}$, $r \in (0, 1)$. Landmark studies \cite{Donoho2004, ingster1997some, ingster1998minimax} have provided the asymptotic \emph{detection boundary} in terms of a function curve of the signal strength and sparsity:
\begin{equation}
r=\rho^{\star}(\alpha) = 
	\left\{
		\begin{array}{lr}
			\alpha - 1/2      		& \quad 1/2 < \alpha \leq 3/4 \\
			(1-\sqrt{1-\alpha})^2 	& \quad 3/4 < \alpha <1.
		\end{array}
	\right.
	\label{equ: detectionBoundary}
\end{equation}
When the signals are below the curve, all tests will fail to distinguish $H_0$ and $H_1$ as $n \rightarrow \infty$. Whenever signals are above the curve, the so-called asymptotically optimal tests are able to make both the type I and the type II error rates converge to zero. A particular optimal test is the Higher Criticism (HC) test \cite{Tukey1976, Donoho2004}: 
\begin{equation}
HC_{n, \mathcal{R}} = \sup_{\mathcal{R}} \sqrt{n} \frac{i/n - p_{(i)}}{\sqrt{p_{(i)}(1-p_{(i)})}},
	\label{equ: HC}
\end{equation}
where $p_{(1)} \leq ... \leq p_{(n)}$ are the order statistics of $p_i=2(1-F_0(|X_i|)), i=1,...,n$, the two-sided p-values of $X_i$. The supremum domain $\mathcal{R}$ regards to the p-value magnitude, or the index $i$, or both. Note that in literature HC formula could also be written as (e.g., \cite{Arias2010, barnett2014analytical} )
\begin{equation}
	HC = \sup_{t \in \mathcal{R}^*} \frac{\sum_i\{ X_i > t\} - 2n\bar{\Phi}(t)}{\sqrt{2n\bar{\Phi}(t)\Phi(t)}}.
	\label{equ: HCstat2}
\end{equation}
Because the formula is monotone, the supremum domain $\mathcal{R}$ on $p_{(i)}$ is equivalent to the supremum domain $\mathcal{R}^*$ on $t$. 

Following that, a variety of modified Higher Criticism (HC), the Berk-Jones  (B-J) type tests, a spectrum of $\phi$-divergence type tests, etc. were studied and proven asymptotically optimal \cite{Donoho2004, Donoho2008, donoho2014higher, jager2007goodness}. In this paper, we provide calculation for the p-value and statistical power for relevant tests under finite $n$. 

\subsection{Study Formulation}

In this paper, we consider arbitrary i.i.d. and continuous null and alternative hypotheses:
\begin{equation}
	H_0: X_i \overset{\text{i.i.d.}}{\sim}  F_0, \quad
	H_a: X_i \overset{\text{i.i.d.}}{\sim}  F_1,  \space i = 1, ..., n.
	\label{equ: generalHypo}
\end{equation}
We consider the goodness-of-fit (GOF) tests that evidence the distinction between the data and any given null distribution.  Consider either one-sided p-values $p_i=1-F_0(X_i)$ or two-sided p-values $p_i=2(1-F_0(|X_i|))$.  When $F_0$ is continuous, the null hypothesis is equivalent to  
\begin{align*}
	&H_0: p_i \overset{\text{i.i.d.}}{\sim} \text{Uniform}(0, 1), \space i = 1, ..., n.
\end{align*}
The typical idea of GOF is that if any $X_i$ do not follow $F_0$, its $p_{(i)}$ shall distinct from it expectation, which is roughly $i/n$. Following this idea we consider a general family of tests, 
\begin{equation}
	S_{n, \mathcal{R}}=\sup_{\mathcal{R}} f(\frac{i}{n}, p_{(i)}),
	\label{equ: GOFstat}
\end{equation}
where the supremum domain 
\begin{equation}
	\mathcal{R}= \{i:  k_0 \leq i \leq k_1 \} \bigcap \{p_{(i)}:  \alpha_0 \leq p_{(i)} \leq \alpha_1 \}
	\label{equ: sup_domain}
\end{equation}
for given $k_0 \leq k_1 \in \{1, ..., n\}$ and $\alpha_0 \leq \alpha_1 \in [0, 1]$. We assume that for fixed $x=i/n$ the function $f(x, y)$ is monotonically decreasing in $y=p_{(i)}$, so that the smaller the input p-values, the larger the statistic and the stronger the evidence against $H_0$. 

This GOF family contains lots of test statistics widely applicable to practice. For example, the simple one-sided Kolmogorov-Smirnov test statistic (c.f. \cite{shaostat},  page 447, denoted $KS^+$ here) is a classic GOF test, which directly measures the difference between $p_{(i)}$ and $i/n$. That is, the $f$ function is defined by
\begin{equation}
	 f_{KS^+}(x,y)=x-y.
	\label{equ: KSstat}
\end{equation}

To improve the test, the difference between $p_{(i)}$ and $i/n$ should be scaled with regard to $p_{(i)}$ or $i/n$. This is because smaller $p_{(i)}$ values are more important to evidence agains $H_0$. Such scaled KS tests are related to the Higher Criticism (HC) statistics proposed in 2004 and 2008 \cite{Donoho2004, Donoho2008}, respectively, where the $f$ functions are defined as
\begin{equation}
	\begin{array}{ll}
	f_{HC^{2004}}(x,y) &=\sqrt{n}\displaystyle\frac{x-y}{\sqrt{y(1-y)}} \text {;          } 	\\
	f_{HC^{2008}}(x,y) &=\sqrt{n}\displaystyle\frac{x-y}{\sqrt{x(1-x)}}. 
	\end{array}
	\label{equ: HCstat}
\end{equation}

Jager and Wellner introduced a collection of $\phi$-divergence statistics \cite{jager2007goodness}, which are also based on the supremum of functions:
\begin{equation}
	\begin{array}{ll}
	f^{\phi}_s(x,y)&=\displaystyle\frac{1}{s(1-s)}(1-x^sy^{1-s}-(1-x)^s(1-y)^{1-s})\text{, } s\neq 0,1,\\
	f^{\phi}_1(x,y)&=x\log(\displaystyle\frac{x}{y})+(1-x)\log(\frac{1-x}{1-y}),\\
	f^{\phi}_0(x,y)&=y\log(\displaystyle\frac{y}{x})+(1-y)\log(\frac{1-y}{1-x}).
	\end{array}
	\label{equ: f_s}
\end{equation}
For certain $s$ (e.g., $s=2$ or $-1$) these statistics are two-sided in the sense that switching the values of $x=i/n$ and $y=p_{(i)}$ gives the same statistic. However, this property is not appropriate in the scenario of signal detection. It is the relationship $p_{(i)} < i/n$, instead of the opposite, indicates signals. Thus for signal detection purpose, one-sided test is more reasonable and more powerful, as is the same idea for defining KS and HC. Thus, here we consider the one-sided version of $\phi$-divergence statistics, which can be achieved by a simple adjustment of the $f$ function to be, for example,
\begin{equation}
f_s(x,y) =
	\left\{
		\begin{array}{lr}
			\sqrt{2nf^{\phi}_s(x,y)} & \quad y \leq x ,\\
			-\sqrt{2nf^{\phi}_s(x,y)} & \quad y > x.
		\end{array}
	\right.
	\label{equ: f_tilde}	
\end{equation}	
Now for all $s$, $f_s(x,y)$ is guaranteed decreasing in y.  Such one-sided $\phi$-divergence statistics cover HC exactly: $f_2=f_{HC^{2004}}$ and $f_{-1}=f_{HC^{2008}}$. Also, $s=1$ and $0$ correspond to the Berk-Jones statistic \cite{berk1979g, Donoho2004, Li2014higher} and the reverse Berk-Jones statistic \cite{berk1979g}, respectively.  Note that the input p-values themselves could be two-sided in order to accommodate the consideration of signal directionality in practical problems.

\subsection{Calculation Strategy}

In the following we first introduce the essential idea of the calculation. Under various settings and assumptions, detailed strategies for obtaining the exact and approximated distributions will be described in Section \ref{Sect: AnaResults}.  

For any given continuous CDFs in (\ref{equ: generalHypo}), we define a monotone transformation function in $(0, 1)$: 
\begin{equation}
	D(x) = \left\{
	\begin{array}{l l}
		x & \quad \text{under $H_0 : F_0$}, \\
		1-F_1(F_0^{-1}(1-x)) & \quad \text{under $H_1 : F_1 \neq F_0$}.
	\end{array} \right.
	\label{equ: D(x)}
\end{equation}
Note that for any p-value $p_i$, we have $D(p_i) \sim Uniform(0, 1)$ under either $H_0$ or $H_1$. 

Consider the function $f(x, y)$ of any statistic $S_{n, \mathcal{R}}$ in (\ref{equ: GOFstat}), for each fixed $x$ define its inverse function 
\begin{equation}
	g(x, \cdot)=f^{-1}(x,\cdot). 
	\label{equ: g(x)}
\end{equation}
For example, for the HC statistics defined (\ref{equ: HCstat}), the $g$ functions are
\begin{equation}
	\begin{array}{ll}
	g_{HC^{2004}}(x,b)=\frac{1}{1+b^2/n}[x+\frac{b^2/n-(b/\sqrt{n})\sqrt{b^2/n+4x(1-x)}}{2}] \text {;          } 	\\
	g_{HC^{2008}}(x,b)=x-(b/\sqrt{n})\sqrt{x(1-x)}. 
	\end{array}
	\label{equ:g2(x)}
\end{equation}
In general if the closed form of $g(x, \cdot)$ is not available, it can always be found numerically, since $f(x,y)$ is strictly decreasing in $y$.

Now under either $H_0$ or $H_1$, the CDF function of $S_{n, \mathcal{R}}$ is  
\begin{equation}
	\begin{array}{l l}	
	P(S_n\leq b)
	&=P(\displaystyle\sup_{\mathcal{R}}f(\frac{i}{n},p_{(i)})\leq b)\\
	&=P(\displaystyle\bigcap_{\mathcal{R}} \{p_{(i)}>g(\frac{i}{n}, b)\})\\
	&=P\{D(p_{(i)})> D(g(\frac{i}{n}, b)), \text{all } i \text{ and } p_{(i)} \text{ in } \mathcal{R} \}. 
	\end{array} 
	\label{equ: P(Sn<b)}
\end{equation}
The key idea is that under either $H_0$ or $H_1$, $U_{(i)}:=D(p_{(i)})$ is the $i^{th}$ order statistic of Uniform(0, 1), and the joint distribution of $U_{(i)}$, $i = k_0, ...., k_1$, can be studies one way or another for getting the final probability.  

To simplify the presentation, we list below the notations to be referred later on. 
\begin{itemize}
\item[\textit{(N1)}] $u_k := D(g(\frac{k}{n}, b)\vee \alpha_0)$, following the definitions in (\ref{equ: D(x)})--(\ref{equ: P(Sn<b)}), and a potential constant $\alpha_0 \geq 0$ in (\ref{equ: sup_domain}). 

\item[\textit{(N2)}] $\bar{F}_{B(\alpha, \beta)}(x)$ denotes the survival function of $Beta(\alpha, \beta)$ distribution.

\item[\textit{(N3)}]  $F_{\Gamma(\alpha)}(x)$ and $\bar{F}_{\Gamma(\alpha)}(x)$ denote the CDF and survival function of $Gamma(\alpha, 1)$ distribution, respectively, where the shape parameter is $\alpha$, the scale parameter is 1.

\item[\textit{(N4)}] Based on the notation \textit{(N3)} define
\[
h_k(x):=xF_{\Gamma(k-1)}(kx)-F_{\Gamma(k)}(kx).
\] 

\item[\textit{(N5)}] $f_{P(\lambda)}(x)$ denotes the probability mass function of $Poisson(\lambda)$ distribution.

\end{itemize}

\section{Analytical Results\label{Sect: AnaResults}}

\subsection{Exact Calculations\label{Sect: ApproxDistn}}

The first theorem provides the exact calculation for the distribution of GOF in (\ref{equ: GOFstat}), where the supremum domain $\mathcal{R}$  involves truncation of the index $i$. For example, the initial HC was defined with $\mathcal{R} = \{ 1 \leq i \leq n/2\}$ \cite{Donoho2004}. 

\begin{theorem}
\label{Thm: exact} 
	Consider any GOF statistic in (\ref{equ: GOFstat}) with $\mathcal{R} = \{ k_0 \leq i \leq k_1\}$ for given $1 \leq k_0 \leq k_1 \leq n$. Let $m=n-k_1+1$. Following the notations (N1), (N2), and  
\[
	\begin{array}{l l}	
		a_{k_1}=\frac{n!}{(n-k_1+1)!}\bar{F}_{B(1, m)}(u_{k_1}) \text{, and } \\
		a_k=\displaystyle \frac{n!}{(n-k+1)!}\bar{F}_{B(k_1-k+1,m)}(u_{k_1}) - \sum_{j=1}^{k_1-k}\frac{u_{k+j-1}^{j}}{j!}a_{k+j}, \quad k=k_1-1, ..., 1.
	\end{array}
\]
Under either $H_0$ or $H_1$ we have 
\[
	P(S_n \leq b)= \bar{F}_{B(k_1,m)}(u_{k_1}) - \sum_{i=k_0}^{k_1-1}\frac{u_i^i}{i!}a_{i+1}. 
\]
\end{theorem}

Another type of truncation is based on the value of $p_{(i)}$. In the case of HC under the null of $N(0,1)$, this truncation is equivalent to the truncation on $t$ in (\ref{equ: HCstat2}). The following theorem gives the exact calculation for the general GOF tests in (\ref{equ: GOFstat}) with such truncations.  

\begin{theorem}
\label{Thm: exacttruncate} 
	Consider any GOF with statistic in (\ref{equ: GOFstat}) with $\mathcal{R} = \{ \alpha_0 \leq p_{(i)} \leq \alpha_1\}$ for given $0 \leq \alpha_0 < \alpha_1 \leq 1$. Following the notations (N1) and (N2), define  
\[
	\begin{array}{l l}	
		\beta_0 = D(\alpha_0), \quad{} \beta_1 = D(\alpha_1), \quad{}
		 c_{ij} = \frac{\beta_0^{i-1}(1-\beta_1)^{n-j+1}}{(i-1)!(n-j+1)!}\\
		a_{j}(k)=\displaystyle \frac{n!}{(j-k)!}\beta_1^{j-k}\bar{F}_{B(j-k,1)}(\frac{u_{j-1}}{\beta_1}) - \sum_{l=1}^{j-k}\frac{u_{k+l-1}^{l}}{l!}a_{j}(k+l), \text{ and } \\
		a_{j}(j)=0,\quad{} 1\leq i\leq k_1,i<j\leq n+1, k=1,...,j-1
	\end{array}
\]
Under either $H_0$ or $H_1$, we have
	\begin{align*}
		P(S_{n,\mathcal{R}} \leq b) = \sum_{i=1}^{k_1}\sum_{j=i+1}^{n+1}c_{ij}a_j(i)
	\end{align*}
\end{theorem}

The most general $\mathcal{R}$ is in (\ref{equ: sup_domain}), which defines the truncation for both the index and the p-values. The following theorem provides the calculation for exact distribution under this general setup, for which no literature has provided solution before. 
\begin{theorem}
\label{Thm: exactgeneral} 
	Consider any GOF with statistic in (\ref{equ: GOFstat})with $\mathcal{R} = \{ \alpha_0 \leq p_{(i)} \leq \alpha_1\} \cap \{ k_0 \leq i \leq k_1\}$ for given $1 \leq k_0 \leq k_1 \leq n$ and  $0 \leq \alpha_0 < \alpha_1 \leq 1$. Following the notations in Theorem \ref{Thm: exacttruncate} and (N1) and (N2), define  
\[
	\begin{array}{l l}	
		\tilde{i}=
		\begin{cases} 
    		 	 i & i\geq k_0 \\
    		 	 k_0 & i<k_0 
		\end{cases}
		,\quad{}
		\tilde{j}=
		\begin{cases} 
			j & j\leq k_1+1 \\
			k_1+1 & j>k_1+1 
		\end{cases}
		, \quad{}
		\tilde{\beta}_0 = \beta_0I_{\{i<k_0\}}, \\
		a_{j}(k)=\displaystyle \frac{n!}{(j-k)!}\beta_1^{(j-k)}\bar{F}_{B(\tilde{j}-k,j-\tilde{j}+1)}(\frac{u_{\tilde{j}-1}}{\beta_1}) - \sum_{l=1}^{\tilde{j}-k}\frac{u_{k+l-1}^{l}}{l!}a_{j}(k+l), \text{ and } \\
		a_{j}(\tilde{j})=0,\quad{} 1\leq i\leq k_1,\tilde{i}<j\leq n+1, k=1,...,\tilde{j}-1.
	\end{array}
\]
Under either $H_0$ or $H_1$, we have
	\begin{align*}
		&P(S_{n,\mathcal{R}} \leq b)\\
		= &\sum_{i=1}^{k_1}\sum_{j=\tilde{i}+1}^{n+1}c_{ij}\left( \frac{n!(\beta_1-\tilde{\beta}_0)^{j-i}}{(j-i)!}\bar{F}_{B(\tilde{j}-i,j-\tilde{j}+1)}(\frac{u_{\tilde{j}-1}-\tilde{\beta}_0}{\beta_1-\tilde{\beta}_0}) - \sum_{k=\tilde{i}}^{\tilde{j}-1}\frac{(u_k-\tilde{\beta_0})^{k-i+1}}{(k-i+1)!}a_{j}(k+1)\right). 
	\end{align*}
\end{theorem}

The study setting in Li and Siegmund \cite{Li2014higher} is a special case of the general truncation with $\alpha_1 = 1$. For example, the ``modified HC" is to take $\mathcal{R} = \{ 1 < i \leq n/2, p_{(i)} \geq 1/n \}$\cite{Donoho2004}. Li and Siegmund's calculation provided an approximation for the modified HC under $H_0$. As shown in Figure \ref{fig:comparison0}, the approximation does not give the whole distribution. For this setting, the corollary below gives the exact distribution for the general GOF in (\ref{equ: GOFstat}) under both $H_0$ and $H_1$.    

\begin{corollary}
\label{Corol: exactmodified}
	Consider any GOF with statistic in (\ref{equ: GOFstat}) with $\mathcal{R} = \{ \alpha_0 \leq p_{(i)} \} \cap \{ k_0 \leq i \leq k_1\}$ for given $1 \leq k_0 \leq k_1 \leq n$ and  $\alpha_0>0$. Following the notations in Theorem \ref{Thm: exact}, \ref{Thm: exacttruncate} and (N1) and (N2), define  
\[
c_{i} = \frac{\beta_0^{i-1}}{(i-1)!}, \quad{}1\leq i \leq k_1. 
\]
Under either $H_0$ or $H_1$, we have 
\begin{align*}
	P(S_{n,\mathcal{R}} \leq b) = \sum_{i=1}^{k_1}c_{i}\left( \frac{n!(1-\tilde{\beta}_0)^{n+1-i}}{(n+1-i)!}\bar{F}_{B(k_1+1-i,m)}(\frac{u_{k_1}-\tilde{\beta}_0}{1-\tilde{\beta}_0}) - \sum_{k=\tilde{i}}^{k_1-1}\frac{(u_{k}-\tilde{\beta}_0)^{k+1-i}}{(k+1-i)!}a_{k+1}\right). 
\end{align*}
\end{corollary}

To sum up, Theorem \ref{Thm: exactgeneral} covers Theorem \ref{Thm: exact} by fixing $i = 1,j = n+1$ (so that $c_{ij} = 0$) and $\alpha_0=0,\alpha_1=1$. It covers Theorem \ref{Thm: exacttruncate} by letting $k_0=1,k_1=n$, and covers Corollary \ref{Corol: exactmodified} by fixing $j = n+1$ and $\alpha_1=1$. 

Regarding the computational complexity, the calculation given by Theorem \ref{Thm: exact} is equivalent to solving an upper triangular linear systems using backward substitution, which is $O(n^2)$, an improvement from $O(n^3)$ required by other exact distribution calculations \cite{noe1972calculation, noe1968calculation, barnett2014analytical}. More importantly, our method can handle a more general family of GOFs defined in (\ref{equ: GOFstat}) with a more flexible supremum domain $\mathcal{R}$. When $\mathcal{R}$ is more complicated, Corollary \ref{Corol: exactmodified}  is still $O(n^2)$ because the inner loop $a_k$ is not dependent on $i$. Theorem \ref{Thm: exacttruncate} and Theorem \ref{Thm: exactgeneral} are $O(n^3)$.

\subsection{Approximate the distributions\label{Sect: ApproxDistn}}

In this section, we provide several calculation methods to approximate the distributions of the GOF family. Under more restricted conditions, computation load could be significantly reduced. For simplicity we provide results here for the supremum domain $\mathcal{R}=\{ k_0 \leq i \leq k_1 \}$. More general results for  $\mathcal{R}$ in (\ref{equ: sup_domain}) can be obtained by following the similar idea of the calculation in Theorem \ref{Thm: exactgeneral}.

The theorem below gives an approximated calculation based on a joint Gamma distribution. %The deduction follows the similar idea of Theorem \ref{Thm: exact}. 

\begin{theorem}
\label{Thm: general} 
	Consider any GOF with statistic in (\ref{equ: GOFstat}) with $\mathcal{R}=\{ k_0 \leq i \leq k_1 \}$. Following the notations (N1) and (N3), define  
\[
\begin{array}{l l}	
d_k = (n+1)D(g(\frac{k}{n}, b)), \quad k=k_0, ..., k_1. \\
c_k=\displaystyle \bar{F}_{\Gamma(k)}(d_{k_1}) - \sum_{j=1}^{k-1}\frac{d_{k_1-k+j}^{j}}{j!}c_{k-j}, \quad k=2, ..., k_1, \text{and } \\
c_{1}=\bar{F}_{\Gamma(1)}(d_{k_1}).
\end{array}
\]
Under either $H_0$ or $H_1$, we have
\[
	P(S_n \leq b)=(1+o(1))  (\bar{F}_{\Gamma(k_1)}(d_{k_1}) - \sum_{k=k_0}^{k_1-1}\frac{d_k^k}{d!}c_{k_1-k}).
\]
\end{theorem}

%Compare with literature, the main value of this theorem is to address a broader class of GOF tests defined in (\ref{equ: GOFstat}) under both $H_0$ and $H_1$. Even if for p-value calculation of related tests under $H_0$, it still shows some relative advantage. In particular, when comparing with \cite{noe1972calculation, noe1968calculation, barnett2014analytical}, which require $k_0=1$ and $k_1=n$ and computation $O(n^3)$, this theorem allows any $1 \leq k_0 \leq k_1 \leq n$ and reduces computation to $O(n^2)$. When comparing with Li and Siegmund \cite{Li2014higher}, which only provides calculation for part of the null distribution, this theorem allows to obtain the whole distribution (see Figures \ref{fig:comparison0} and \ref{fig:comparison} for comparison).    

To further reduce the computation, we can also deduce a one-step formula under stronger assumptions. In particular, if $D(g(\frac{k}{n}, b))$ is a linear function of $k$, we can provide a closed-form formula for the distribution that gives the same accuracy as the above theorem. 

\begin{theorem}
\label{Thm: linearCase} 
	Consider any GOF with statistic in (\ref{equ: GOFstat}) with $\mathcal{R}=\{ 1 \leq i \leq k_1 \}$. Assume the function $D(g(\frac{k}{n}, b))=a + \lambda k$, for some $\lambda \geq 0$. Following the notation (N3) and (N4), under either $H_0$ or $H_1$ we have   
\[
       P(S_n \leq b)=(1+o(1))e^{-a}(1-\lambda+h_{k_1}(\lambda)). 
\]
\end{theorem}

\smallskip
One example that satisfies the linearity of $D(g(\frac{k}{n}, b))$ is the simple Kolmogorov-Smirnov ($KS^+$) test in (\ref{equ: KSstat}) under $H_0$, where $a=-(n+1)b$ and $\lambda=\frac{n+1}{n}$. In this case, based on the above theorem we have the following corollary
\begin{corollary}
\label{Corol: KS}
\emph{}
	Consider the test statistic $KS^+$ in (\ref{equ: KSstat}) under $H_0$. Following the notation (N3) and (N4), for $b\leq \frac{1}{n}$, we have 
\[
	P(KS^+ \leq b)=(1+o(1))e^{(n+1)b}(-\frac{1}{n}+h_{k_1}(\frac{n+1}{n})).  
\]
\end{corollary}

\smallskip
The requirement of linear $D(g(\frac{k}{n}, b))$ in Theorem \ref{Thm: linearCase} is stringent. However, if $D(g(\frac{k}{n}, b))$ is close to linear, we can still further simplify the calculation of Theorem \ref{Thm: general}. As one example, Theorem \ref{Thm: CloseLinearCase} below provides such a calculation formula for GOF and the conditions on $D(g(\frac{k}{n}, b))$. The strategy for proving this theorem largely follows the idea in Li and Siegmund  \cite{Li2014higher}. The key difference is that, instead of using the beta distribution in Li and Siegmund, we use the gamma distribution, which has a simpler density function for addressing a wider family of the GOF tests. See Appendix for details. 
 
 \begin{theorem}
\label{Thm: CloseLinearCase}
\emph{}
	Consider GOF tests in (\ref{equ: GOFstat}) with $\mathcal{R}=\{ k_0 \leq i \leq k_1 \}$. Follow notations (N1) -- (N5), and define $d_{k}=(n+1)D(g(\frac{k}{n}, b))$, $d_{k}^\prime=(n+1)\frac{d }{d x}D(g(\frac{k}{n}, b))$, and $k^*=\min\{k_1-k,\sqrt{n}\}$. Assume $D(g(x, b))$ satisfies
\begin{enumerate}
\item $D(g(x, b)) < 1$ is increasing and convex in x for $\frac{k_0}{n}\leq x\leq \frac{k_1}{n}$,
\item $\frac{d}{d x}D(g(x, b))<1$, and
\item $D(g(k/n, b))<\frac{k}{n+1}$, for $k>1$ and large $n$. 
\end{enumerate}
Under either $H_0$ or $H_1$ we have
\[
       P(S_n\geq b)=(1+o(1))\sum_{k=k_0}^{k1}(1-\frac{d_k^\prime}{n}+h_{k^*}(\frac{d_k^\prime}{n}))f_{P(d_k)}(k).
\]
\end{theorem}

\smallskip
Such a close-to-linear property of $D(g(\frac{k}{n}, b))$ can be satisfied by $HC^{2004}$ under $H_0$. In this case, $D(g(\frac{k}{n}, b)) = g(x, b)$ is given in (\ref{equ:g2(x)}), and the conditions can be satisfied when $b$ is in the order of $O(\sqrt{n})$. Thus, we can get a computation-easy formula. 

\begin{corollary}
\label{Corol: HC}
\emph{}
	Consider the test statistic $HC^{2004}$ in (\ref{equ: HCstat}) with $\mathcal{R}=\{ k_0 \leq i \leq k_1 \}$. Let $b_0=\frac{b}{\sqrt{n}}$ be a positive constant $>2x-1$, $\frac{k_0}{n}<x<\frac{k_1}{n}$. Define function 
	\begin{align*} 
	g(x,b_0) &= \frac{1}{1+b_0^2}[x+(b_0^2-b_0\sqrt{b_0^2+4x(1-x)})/2], \\
	g^\prime(x, b_0) &= \frac{1}{1+b_0^2}[1-\frac{b_0(1-2x)}{\sqrt{b_0^2+4x(1-x)}}].
	\end{align*} 
Following the notation (N2), under $H_0$, we have
\[
	P(HC^{2004}\geq b) =(1+o(1)) \sum_{k=k_0}^{k1}\left(1-g^\prime(\frac{k}{n}, b_0)+h_{k^*}(g^\prime(\frac{k}{n}, b_0))\right)f_{P(g(\frac{k}{n}, b_0)n)}(k). 
\]
\end{corollary}

The above formula is different from that given in Li and Sigmund  \cite{Li2014higher}.  However, for the HC under $H_0$, both formulas require the threshold $b=O(\sqrt{n})$. Thus, both methods do not get the whole distribution. Furthermore, the accuracy depends on the linear approximation of the $D(g(\frac{k}{n}, b))$ function, which could be far from the truth under general $H_1$. Thus this type of calculation may not be very attractive for calculation statistical power. Meanwhile, numeric results in Section \ref{Sect: NumResults} show that both methods are accurate for calculating small p-values under relatively large threshold $b$.

\section{Numerical results\label{Sect: NumResults}}

In this section we first evidence the accuracy of our methods by comparing the calculations with the Monte-Carlo simulations under various settings of $H_0$ and $H_1$. Then, based on calculation we compare the finite-$n$ performance of the asymptotically optimal tests over various signal patterns. The supremum domain is $\mathcal{R} = \{ 1 \leq i \leq n/2\}$ by default, and will be specified otherwise. The simulations run 5,000 repetitions by default. 

\subsection{Evaluate the accuracy of calculations}
Under $H_0: X_i \overset{\text{i.i.d.}}{\sim}  N(0, 1) , \space i = 1, ..., n$. Figure \ref{fig:comparison} shows the right-tail probability of HC over the threshold $b$. Comparing with simulation (black solid curves), the exact calculation by Theorem \ref{Thm: exact} (cyan dashed curves) has a perfect match. The approximation by Theorem \ref{Thm: general} is accurate over the whole distribution when $n$ is fairly big. The closed-formula calculation methods by Li and Siegmund \cite{Li2014higher} (blue dotted curves) and by Corollary \ref{Corol: HC} (green dashed curves) do not fit the whole distribution curve, but both can provide good approximation for calculation small p-values at large threshold.  

%\lipsum[1-3]
%\afterpage{
\begin{figure}
	\caption{Comparison among different calculations for the distributions of HC over various $n$ value under $H_0: X_i \overset{\text{i.i.d.}}{\sim} N(0, 1)$. Black solid curves: by simulations; cyan dashed curves: by Theorem \ref{Thm: exact} for exact distribution; red dot-dashed curves: by Theorem \ref{Thm: general};  blue dotted curves: by Li and Siegmund \cite{Li2014higher}; green dashed curves: by Corollary \ref{Corol: HC} \label{fig:comparison}.  }
	\begin{center}
		\subfloat{\includegraphics[width=0.5\textwidth,height=2in]{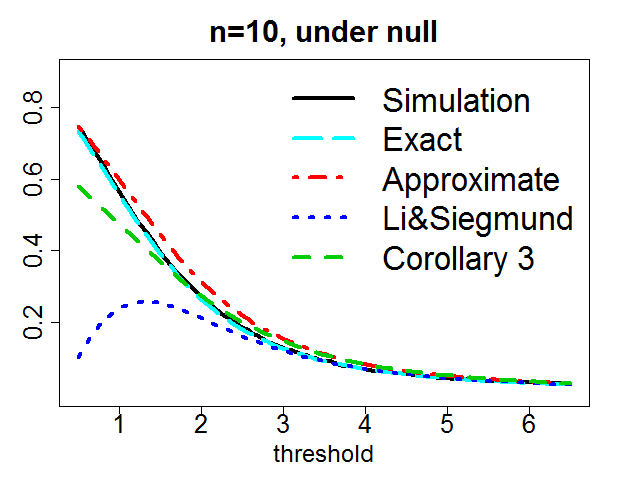}}
		\subfloat{\includegraphics[width=0.5\textwidth,height=2in]{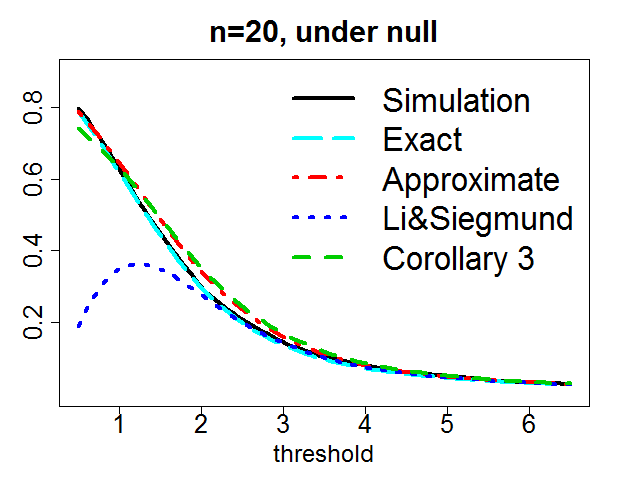}}\\
		\subfloat{\includegraphics[width=0.5\textwidth,height=2in]{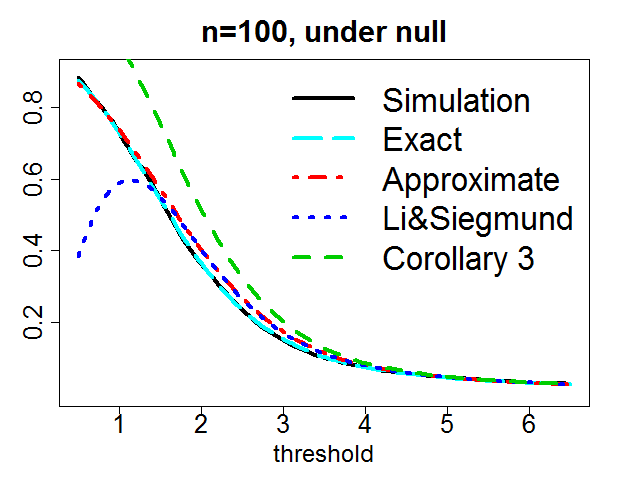}}
		\subfloat{\includegraphics[width=0.5\textwidth,height=2in]{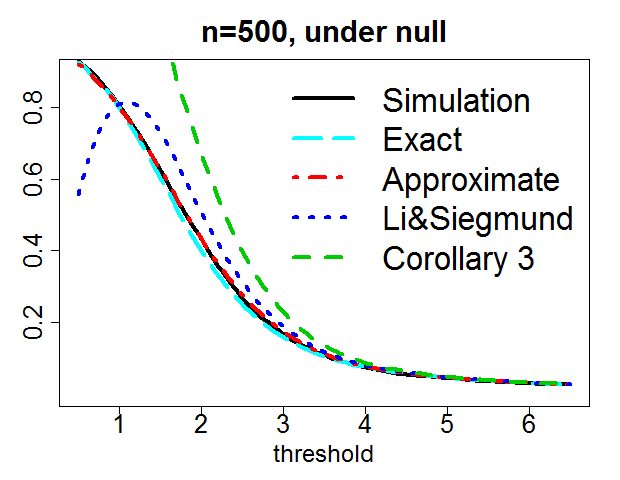}}\\
	\end{center}
\end{figure}
%}
%\lipsum[4-6]

For distribution calculation under the alternative, we considered the null of i.i.d. $N(0,1)$ (on which the input p-values are based) and the alternatives: 
\begin{align*}
	&H_1: X_i \overset{\text{i.i.d.}}{\sim} (1-\epsilon)N(0,1)+\epsilon N(1,1), \text{ or}\\
	&H_1: X_i \overset{\text{i.i.d.}}{\sim} (1-\epsilon)N(0,1)+\epsilon t_{\nu}. 	
\end{align*}
Figure \ref{fig:fitnormal} shows the right-tail probability of HC under the normal mixture alternative with $\mu=1$, $\epsilon=0.1$ (row 1), or under the Student's $t$ distribution with degrees of freedom $\nu=5$ (row 2). In both cases the distribution curves by exact calculation (Theorem \ref{Thm: exact}, cyan dashed curves) and by approximation (Theorem \ref{Thm: general}, red dot-dashed) are close to simulation (black solid curves). As expected, the approximation is more accurate for larger $n$. 
 
\begin{figure}
	\caption{The alternative distributions of HC under $H_0: X_i \overset{\text{i.i.d.}}{\sim} N(0, 1)$ vs. $H_1: X_i\sim 0.9N(0,1)+0.1N(1,1)$ (row 1), or $H_1: X_i\sim 0.5N(0,1)+ 0.5t_{5}$ (row 2). Column 1: $n=10$; column 2: $n=100$. Black solid curves: by simulations; cyan dashed curves: by Theorem \ref{Thm: exact}; red dot-dashed curves: by Theorem \ref{Thm: general}. } \label{fig:fitnormal} 
	\begin{center}
		\subfloat{\includegraphics[width=2.5in,height=2in]{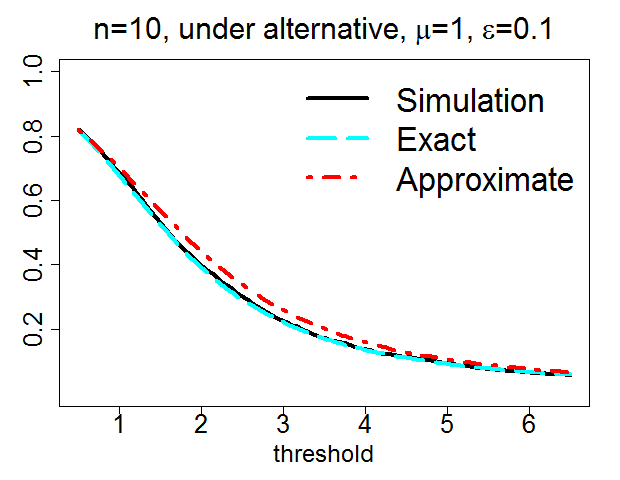}}
		\subfloat{\includegraphics[width=2.5in,height=2in]{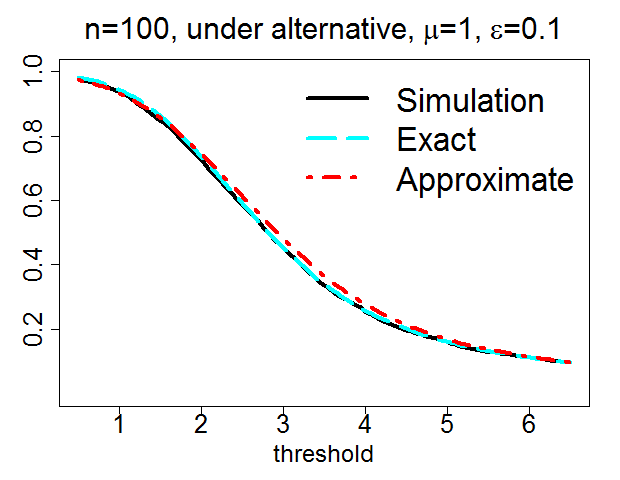}}\\
		%\subfloat{\includegraphics[width=2.5in,height=2in]{alter500.png}}\\
		\subfloat{\includegraphics[width=2.5in,height=2in]{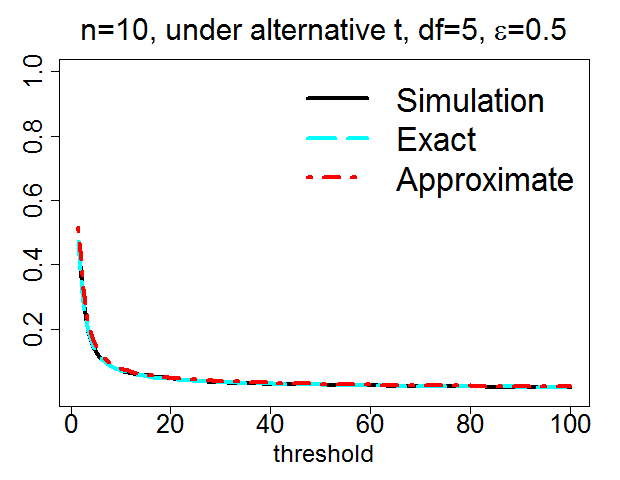}}
		\subfloat{\includegraphics[width=2.5in,height=2in]{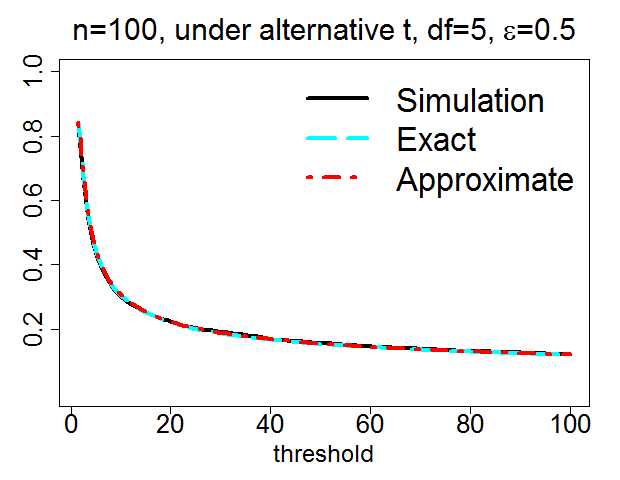}}\\
		%\subfloat{\includegraphics[width=2.5in,height=2in]{alter500tdf5.png}}\\
	\end{center}
\end{figure}

Besides the normal distributions, Theorems \ref{Thm: exact} -- \ref{Thm: general} can handle any given continuous $F_0$ and $F_1$. Here we test on four settings studied in the initial paper of HC \cite{Donoho2004}. The first setting regards a Chi-squared model:
\begin{align*}
	H_0: X_i \overset{\text{i.i.d.}}{\sim} \chi^2_{\nu}(0), \quad \text{   vs.  } \quad
	H_1: X_i \overset{\text{i.i.d.}}{\sim} (1-\epsilon)\chi^2_{\nu}(0)+\epsilon\chi^2_{\nu}(\delta). 
\end{align*}
where $\nu$ is the degree of freedom, $\delta$ is the non-centrality parameter. The second setting is a Student's t mixture model:
\begin{align*}
	H_0: X_i \overset{\text{i.i.d.}}{\sim} t_{\nu}(0), \quad \text{   vs.  } \quad
	H_1: X_i \overset{\text{i.i.d.}}{\sim} (1-\epsilon)t_{\nu}(0)+\epsilon t_{\nu}(\delta).
\end{align*}
The third setting is a chi-squared-exponential mixture model,
\begin{align*}
	H_0: X_i \overset{\text{i.i.d.}}{\sim} \exp (\nu), \quad \text{   vs.  } \quad
	H_1: X_i \overset{\text{i.i.d.}}{\sim} (1-\epsilon)\exp(\nu)+\epsilon\chi^2_{\nu}(\delta).
\end{align*}
The fourth setting concerns a generalized normal distribution (also known as power exponential distribution) model, 
\begin{align*}
	H_0: X_i \overset{\text{i.i.d.}}{\sim} GN_{p}(0,\sigma), \quad \text{   vs.  } \quad
	H_1: X_i \overset{\text{i.i.d.}}{\sim} (1-\epsilon)GN_{p}(0,\sigma) + \epsilon GN_{p}(\mu,\sigma),
\end{align*}
where the probability density function of $GN_{p}(\mu,\sigma)$  is  
\begin{align*}
	\frac{1}{C_{p}}\exp(-\frac{|x-\mu|^p}{p\sigma^p}),\quad C_{p}=2p^{1/p}\Gamma(1+1/p)\sigma. 
\end{align*}
Notice that $GN_{1}(\mu,\sigma)$ is the Laplace distribution and $GN_{2}(\mu,\sigma)$ is $N(\mu,\sigma^2)$. Each row of Figure \ref{fig:fitdifferent} illustrates the alternative distribution of HC under each of the four settings for $n=10$ (left column) and $100$ (right column). Again the distribution calculation is fairly accurate in all cases. 

\begin{figure}
	\caption{The alternative distributions of HC under four non-normal settings for $H_0$ and $H_1$. Column 1: $n=10$; column 2: $n=100$. Black solid curves: by simulations; cyan dashed curves: by Theorem \ref{Thm: exact}; red dot-dashed curves: by Theorem \ref{Thm: general}. }\label{fig:fitdifferent}
	\begin{center}
		\subfloat{\includegraphics[width=2.5in,height=1.7in]{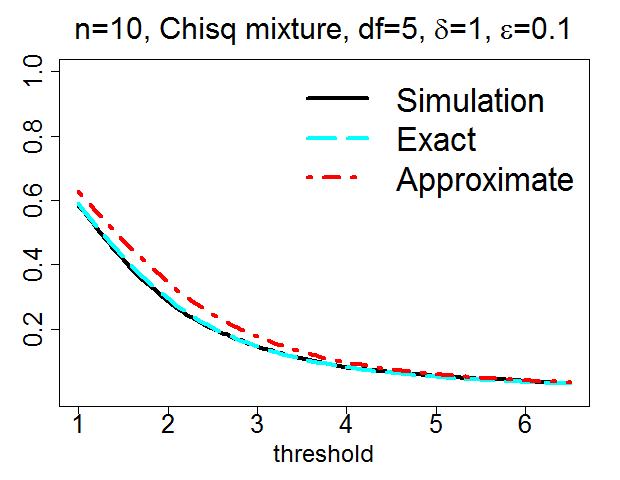}}
		\subfloat{\includegraphics[width=2.5in,height=1.7in]{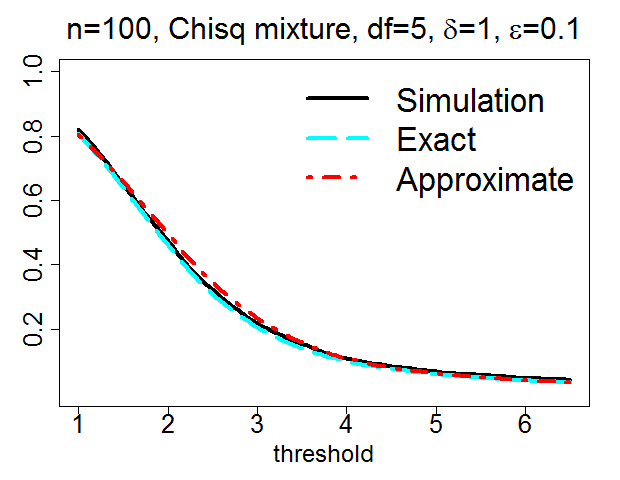}}\\
		%\subfloat{\includegraphics[width=2.5in,height=1.7in]{alter500chisqdf5.png}}\\
		\subfloat{\includegraphics[width=2.5in,height=1.7in]{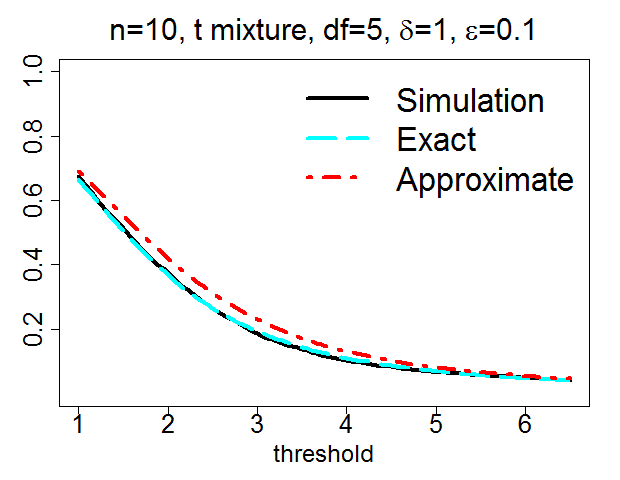}}
		\subfloat{\includegraphics[width=2.5in,height=1.7in]{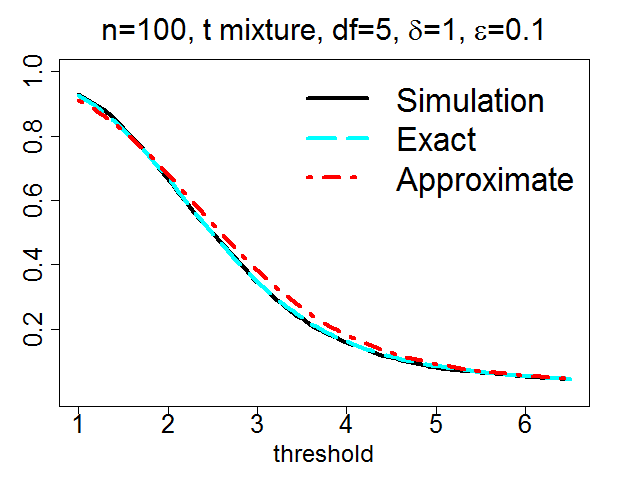}}\\
		%\subfloat{\includegraphics[width=2.5in,height=1.7in]{alter500ttdf5.png}}\\
		\subfloat{\includegraphics[width=2.5in,height=1.7in]{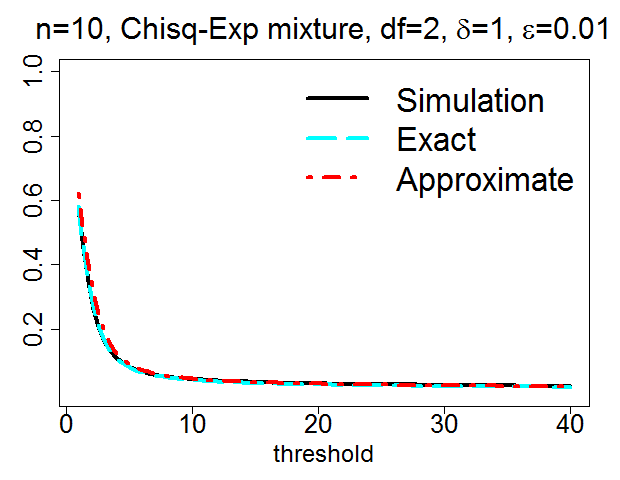}}
		\subfloat{\includegraphics[width=2.5in,height=1.7in]{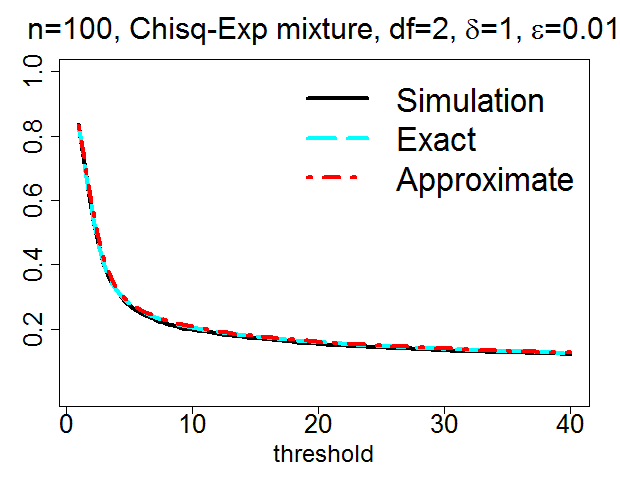}}\\
		%\subfloat{\includegraphics[width=2.5in,height=1.7in]{alter500chisqexpdf2eps001.png}}\\
		\subfloat{\includegraphics[width=2.5in,height=1.7in]{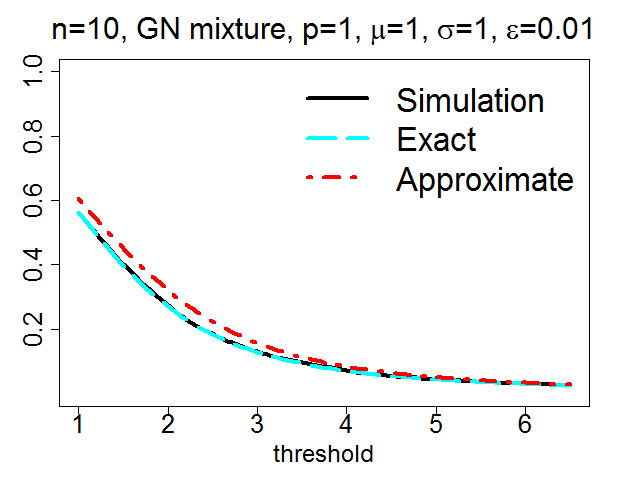}}
		\subfloat{\includegraphics[width=2.5in,height=1.7in]{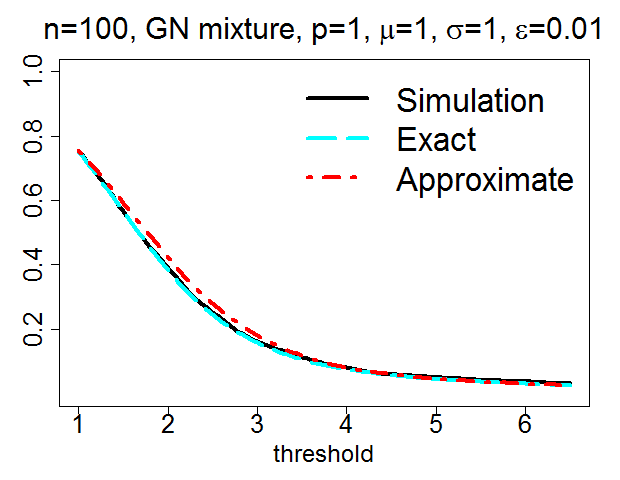}}\\
		%\subfloat{\includegraphics[width=2.5in,height=1.7in]{alter500gnp1.png}}\\
	\end{center}
\end{figure}

For the closed-form calculation formula given by Theorem \ref{Thm: linearCase}, the boundary is assumed linear: $D(g(\frac{i}{n}, b))= a+\lambda k\geq 0 $ in (\ref{equ: P(Sn<b)}). One example is the $KS^+$ in (\ref{equ: KSstat}) under $H_0$. Figure \ref{fig:beta} demonstrates the accuracy of the calculation based on either fixed slope $\lambda=0.5$ or fixed intercept $a=0.5$. Here $k_0=1$, $k_1=n=50$. As the boundary $a+\lambda k$ increases, the probabilities from both calculation and simulation decrease and well-matched as expected. 

\begin{figure}
	\caption{Probability in (\ref{equ: P(Sn<b)}) with boundary $D(g(\frac{i}{n}, b))=a+\lambda k$. Black solid curves: by simulations; red dot-dashed curves: by Theorem \ref{Thm: linearCase}.}\label{fig:beta}
	\begin{center}
		\subfloat{\includegraphics[width=2.5in,height=2in]{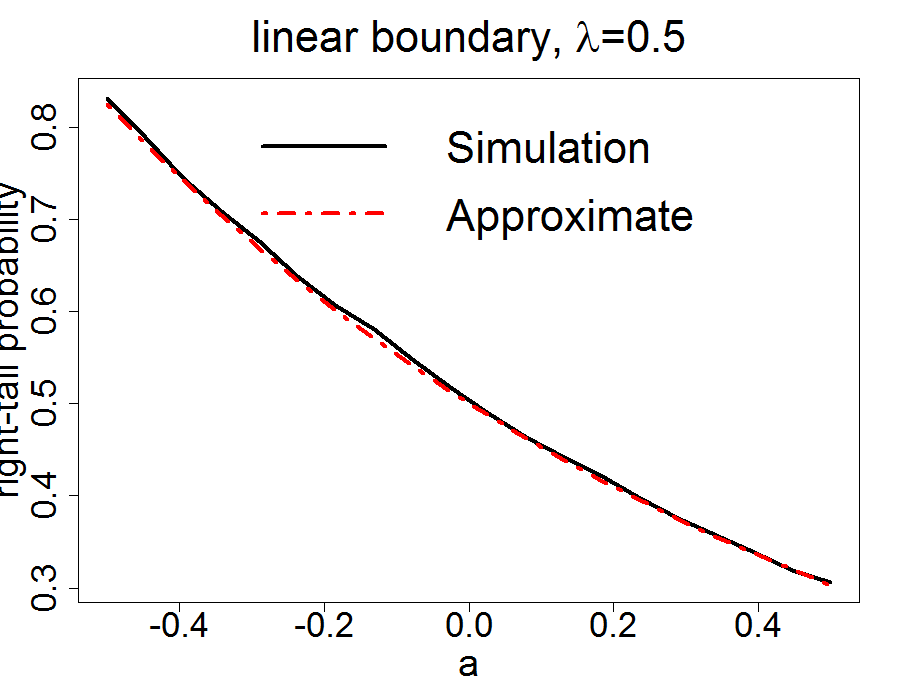}}
		\subfloat{\includegraphics[width=2.5in,height=2in]{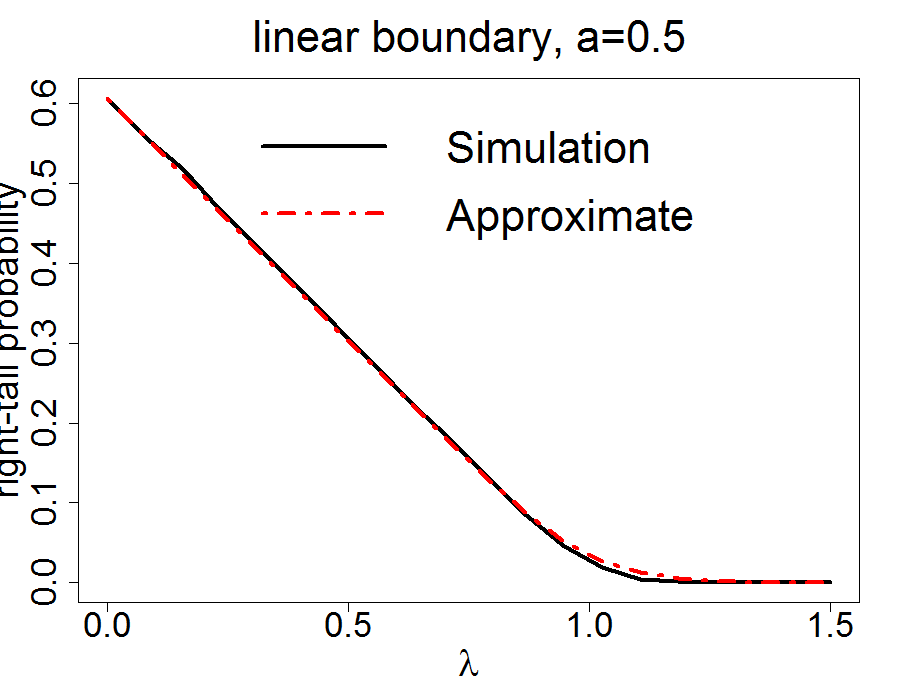}}
	\end{center}
\end{figure}

%Figure \ref{fig:KS} shows the distribution of $KS^+$ in (\ref{equ: KSstat}) for $n=$ 10 and 500 under null hypothesis. Corollary \ref{Corol: KS} is valid for threshold on the x-axis less than $\frac{1}{n}$. 
%\begin{figure}
%	\caption{Distribution of $KS^+$ under null hypothesis.}\label{fig:KS}
%	\begin{center}
%		\subfloat{\includegraphics[width=2.5in,height=2in]{ks10a.png}}%Note:  Truncate at x=0.15
%		\subfloat{\includegraphics[width=2.5in,height=2in]{ks500a.png}} %Note:  Truncate at x=0.0025
%		%Note: Remove title term ", null ... of KS+"
%	\end{center}
%\end{figure}

\subsection{Compare statistical power of asymptotically optimal tests}

All $\phi$-divergence statistics with $s \in [-1, 2]$ possess the same asymptotic optimality property for detecting weak and sparse signals \cite{jager2007goodness}. It is of interest to know the performance of such statistics under finite $n$. Here we study $s=2,1,0,-1$, which correspond to $HC^{2004}$, the Berk-Jones statistic, the reverse Berk-Jones statistic, and $HC^{2008}$, respectively. 

To show the calculation accuracy, we calculated (by Theorem \ref{Thm: exact}) the critical values at the significance levels of $10\%$, $5\%$ and $1\%$. Then at these critical values we got the empirical type I error rates through simulation (10,000 repetitions). Table \ref{tab:pvalue_exact} presents the thresholds by calculation and the correspondingly empirical type I error rates. The closeness of these empirical and nominal type I error rates evidences that the calculation for these tests is accurate. 

\begin{table}
	\caption {Type I error rates: exact calculation vs. simulation. } \label{tab:pvalue_exact} 
	\begin{tabular}{lllllllll}
		\toprule 
		    s & n& \multicolumn{2}{c}{10\%} & \multicolumn{2}{c}{5\%}& \multicolumn{2}{c}{1\%}\\
		      &  & threshold& simulation&threshold& simulation&threshold& simulation\\
		\midrule
		2	&10	&3.357	&0.992  &4.648 &0.049 	&10.088	 &0.010 \\
			&50	&3.507	&0.102  &4.714 &0.050 	&10.102	 &0.011 \\
		 	&100	&3.539	&0.103  &4.723 &0.049 	&10.102	 &0.009 \\
			%&500	&3.563	&0.098  &4.731 &0.049 	&10.102	 &0.011 \\
		1	&10	&2.181 &0.101	 &2.504 &0.050 	&3.110	 &0.011 \\
			&50	&2.408 &0.098	 &2.716 &0.048 	&3.300	 &0.010 \\
			&100	&2.478	&0.104	 &2.780 &0.049 	&3.354	 &0.009\\
			%&500	&2.535 &0.117 &2.821	 &0.060 	&3.375	 &0.012 \\
		0	&10	&1.750 &0.100	 &1.974 &0.049 	&2.390	 &0.011 \\
			&50	&2.040	&0.101	&2.301	 &0.047	&2.803 &0.011 \\
			&100	&2.136	&0.101 &2.402	 &0.051 	&2.915 &0.010 \\
			%&500	&2.220	&0.113 &2.465	 &0.066 	&2.933 &0.014 \\
		-1	&10	&1.618 &0.098	 &1.838 &0.051 	&2.227	 &0.009 \\
			&50	&1.909	&0.099	&2.165  &0.049 	&2.662 &0.009 \\
		 	&100	&2.010 &0.107 &2.271	 &0.052 	&2.777	 &0.010 \\
			%&500	&2.091 &0.118 &2.323 &0.065 	&2.760	 &0.021 \\
		\bottomrule
	\end{tabular}
\end{table}

Now through the calculation based on Theorem \ref{Thm: exact}, we systematically compared the power of these tests under 
\begin{align*}
	H_0: X_i \overset{\text{i.i.d.}}{\sim} N(0, 1), \quad \text{   vs.  } \quad
	H_1: X_i \overset{\text{i.i.d.}}{\sim} (1-\epsilon)N(0, 1)+\epsilon N(\mu, 1). 
\end{align*}
With the type I error rate controlled at 5\%, Figure \ref{fig:power} provides the statistical power at various $\mu$, $n$ and $\epsilon$. There are a few interesting observations. First, $HC^{2004}$ performs well when signals are sparse. However, it seems more relevant to the average number of signals, i.e., $\epsilon n$, rather than the proportion $\epsilon$ itself. For example, at fixed $\epsilon n = 5$ (panels in the first column), $HC^{2004}$ is always the best, while at fixed $\epsilon = 0.05$ (panels in the diagonal), $HC^{2004}$ becomes worse when $n$ increases. 
Second, Smaller $s$ parameter ($s=-1$ for $HC^{2008}$ and $s=0$ for the reverse Berk-Jones) gives better power for denser signals. However, when signals are sparse, they are less powerful because these statistics are insensitive to small p-values as they are weighted by $i/n$ rather than $p_{(i)}$.   
Third, Berk-Jones statistic has a more robust performance over various $\mu$, $n$ and $\epsilon$, which is consistent with the finding of Li and Siegmund \cite{Li2014higher}. 

\begin{figure}
	\caption{Comparison of statistical power (at type I error rate 5\%).}\label{fig:power}
	\subfloat{\includegraphics[width=0.33\textwidth,height=1.8in]{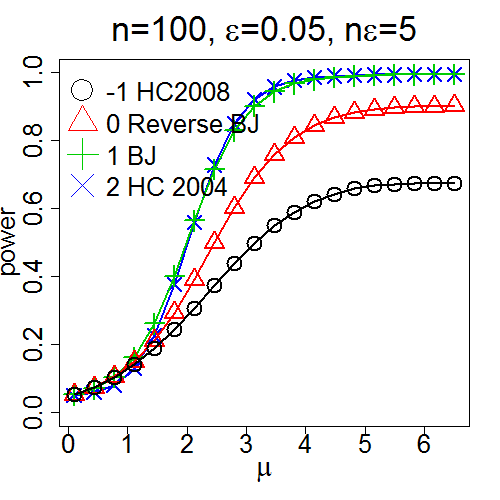}}
	\subfloat{\includegraphics[width=0.33\textwidth,height=1.8in]{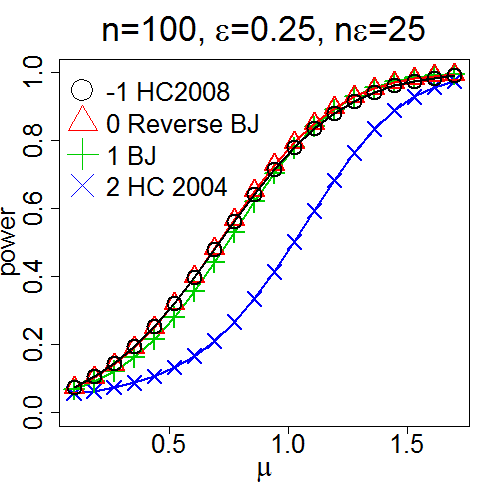}}
	\subfloat{\includegraphics[width=0.33\textwidth,height=1.8in]{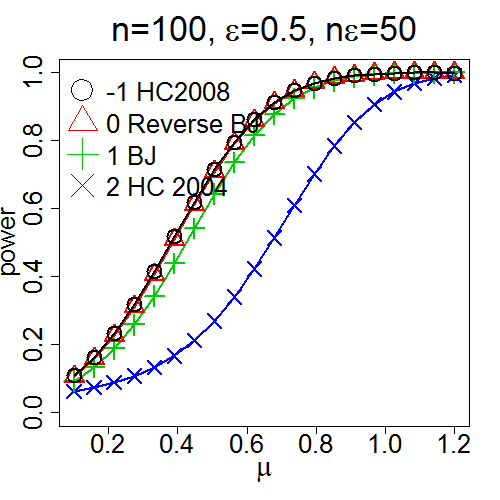}}\\
	\subfloat{\includegraphics[width=0.33\textwidth,height=1.8in]{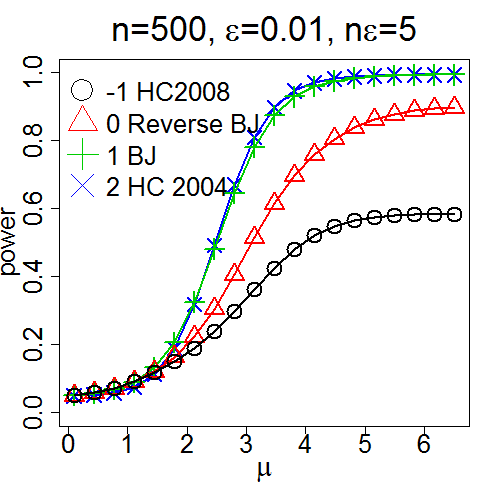}}
	\subfloat{\includegraphics[width=0.33\textwidth,height=1.8in]{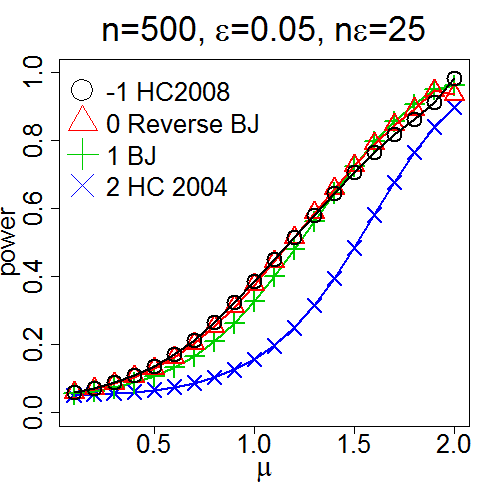}}
	\subfloat{\includegraphics[width=0.33\textwidth,height=1.8in]{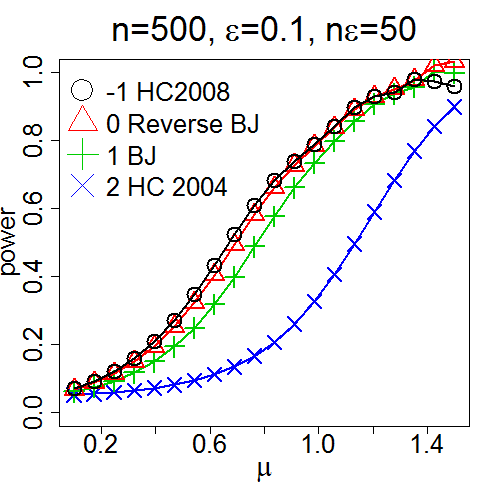}}\\
	\subfloat{\includegraphics[width=0.33\textwidth,height=1.8in]{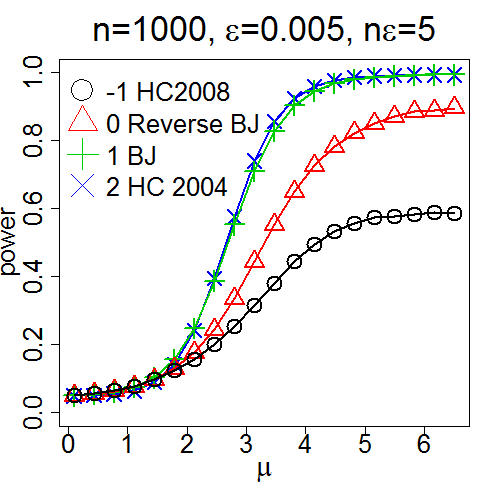}}
	\subfloat{\includegraphics[width=0.33\textwidth,height=1.8in]{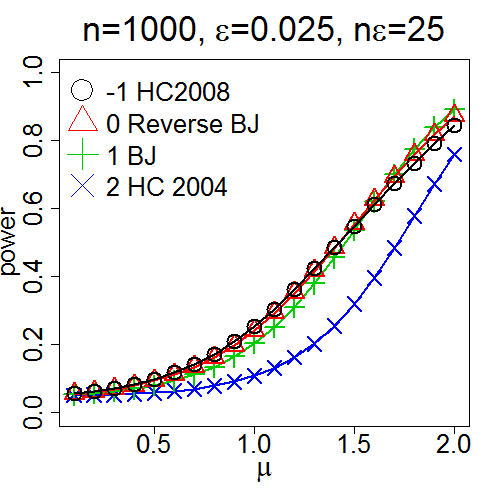}}
	\subfloat{\includegraphics[width=0.33\textwidth,height=1.8in]{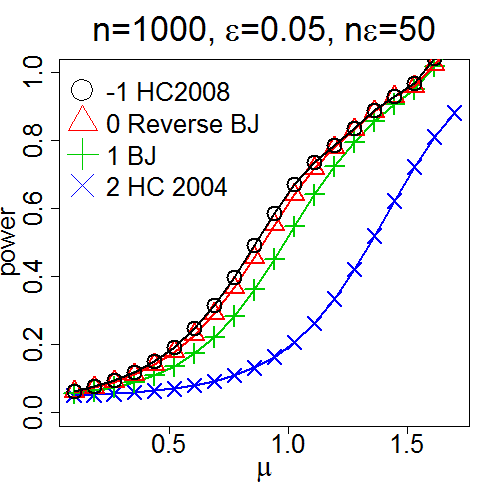}}\\
	\subfloat{\includegraphics[width=0.33\textwidth,height=1.8in]{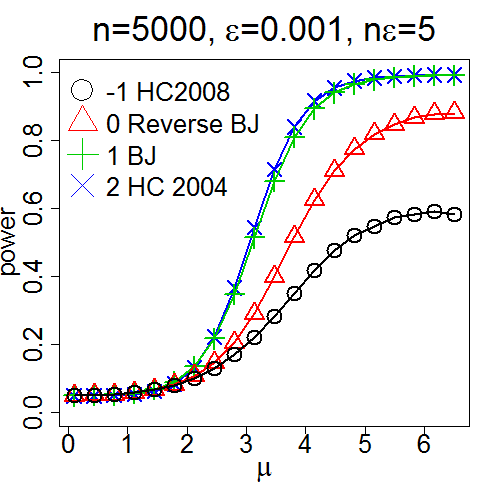}}
	\subfloat{\includegraphics[width=0.33\textwidth,height=1.8in]{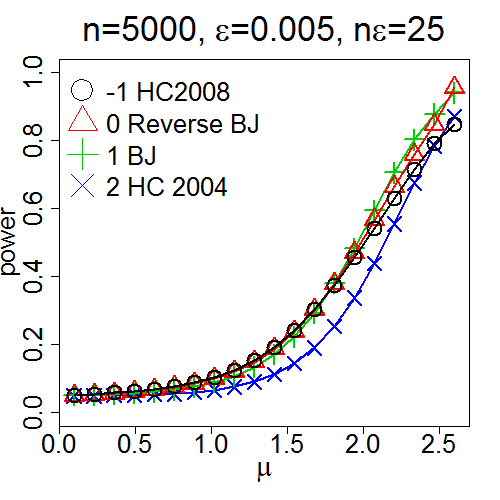}}
	\subfloat{\includegraphics[width=0.33\textwidth,height=1.8in]{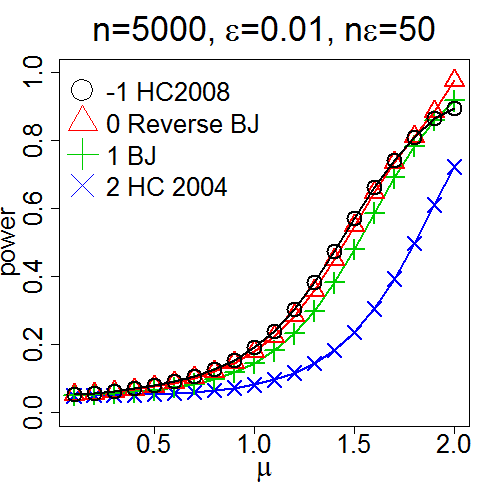}}
\end{figure}

Moreover, it is interesting to see the performance of these methods along the asymptotic detection boundary. This is because the optimal tests are most  valuable for detecting subtle signals around the detection boundary, for which sub-optimal tests will have asymptotically zero power. Here we consider the ARW setting with the detection boundary given in (\ref{equ: detectionBoundary}). Figure \ref{fig:power_detect} shows the power (calculated by Theorem \ref{Thm: exact}) of the four methods over the sparsity parameter $\alpha \in (1/2, 1)$ and $r$ calculated accordingly. It shows that the power curves of $HC^{2004}$ and B-J have the similar pattern, while the performance of $HC^{2008}$ and reverse B-J are alike. In the less sparse case with $\alpha \in (1/2, 3/4)$, B-J is preferred; in the sparser case with $\alpha \in (3/4, 1)$, $HC^{2004}$ outperforms the others.

%Figure \ref{fig:powerapprox_detect} is calculated by approximation Theorem \ref{Thm: general}, and is similar as Figure \ref{fig:power_detect} calculated by exact method Theorem \ref{Thm: exact}. 
%\begin{figure}
%	\caption{Approximated ower at 5\% significance along the detection boundary.}\label{fig:powerapprox_detect}
%	\begin{center}
%		\subfloat{\includegraphics[width=2.5in,height=2in]{detect500approx.png}}
%		\subfloat{\includegraphics[width=2.5in,height=2in]{detect50000approx.png}}
%	\end{center}
%\end{figure}

\begin{figure}
	\caption{Statistical power along the ARW detection boundary (at type I error rate 5\%). }\label{fig:power_detect}
	\begin{center}
		\subfloat{\includegraphics[width=2.5in,height=2in]{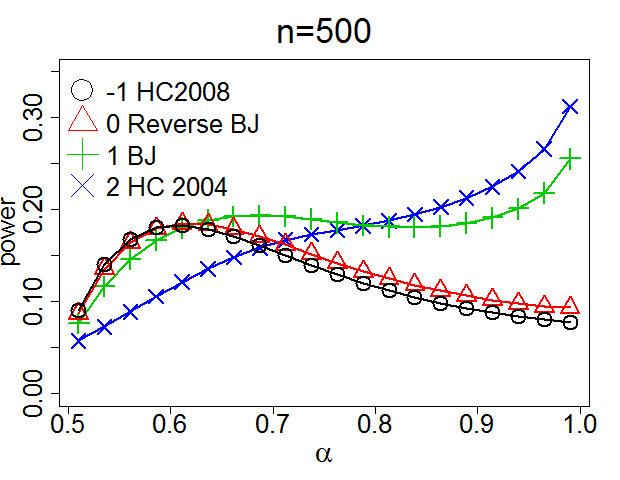}}
		\subfloat{\includegraphics[width=2.5in,height=2in]{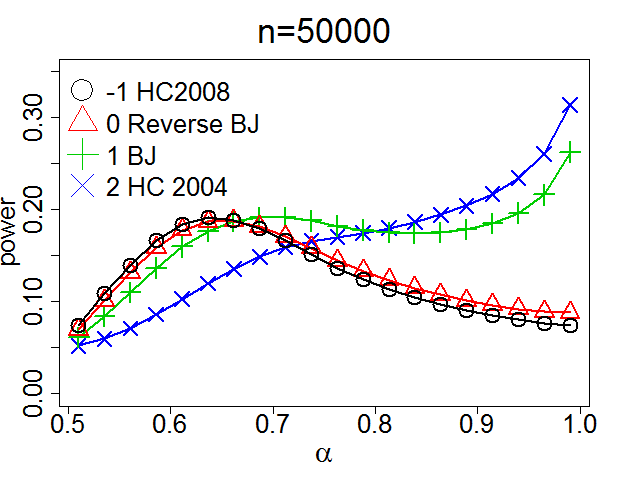}}
	\end{center}
\end{figure}

Last but not least, the supremum domain $\mathcal{R}$ in (\ref{equ: GOFstat}) is quite relevant in constructing the test statistic. Here we compare $HC^{2004}$ under $\mathcal{R} = \{ 1 \leq i \leq n/2\}$ with the modified HC under $\mathcal{R} = \{ 1 < i \leq n/2, p_{(i)} \geq 1/n \}$. 
Figure \ref{fig:power_mhc} shows that the MHC performs poorly when the number of signals is small, whereas it improves the performance when the number of signals increases. When signals are sparse, MHC is less powerful because it tends to exclude signals by considering only those p-values bigger than $1/n$, the latter is a fairly large value when $n$ is small. Thus, in practice when $n$ is not too big, there may be no need to truncate p-values by $1/n$.   

\begin{figure}
	\caption{Power comparison for HC vs. MHC (at type I error rate 5\%).}\label{fig:power_mhc}
	\begin{center}
%%These two graphs are comparison for 5 methods together. 
%		\subfloat{\includegraphics[width=0.33\textwidth,height=1.8in]{1005modifiedexact.png}}
%		\subfloat{\includegraphics[width=0.33\textwidth,height=1.8in]{10025modifiedexact.png}}
%		\subfloat{\includegraphics[width=0.33\textwidth,height=1.8in]{10050modifiedexact.png}}\\
		\subfloat{\includegraphics[width=0.33\textwidth,height=1.8in]{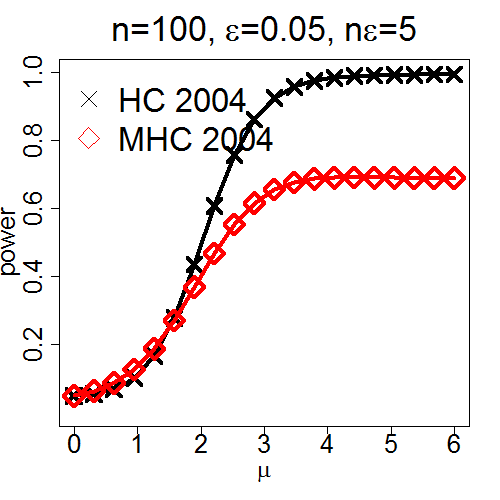}}
		\subfloat{\includegraphics[width=0.33\textwidth,height=1.8in]{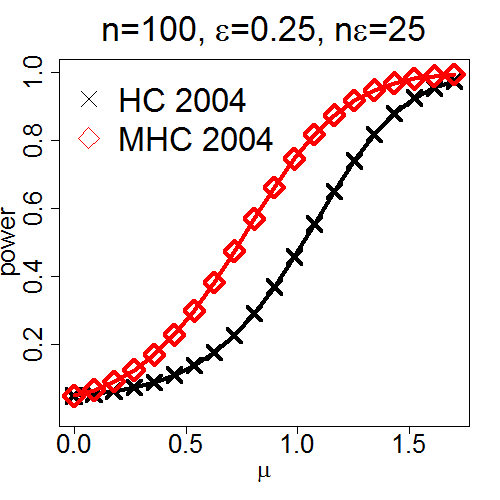}}
		\subfloat{\includegraphics[width=0.33\textwidth,height=1.8in]{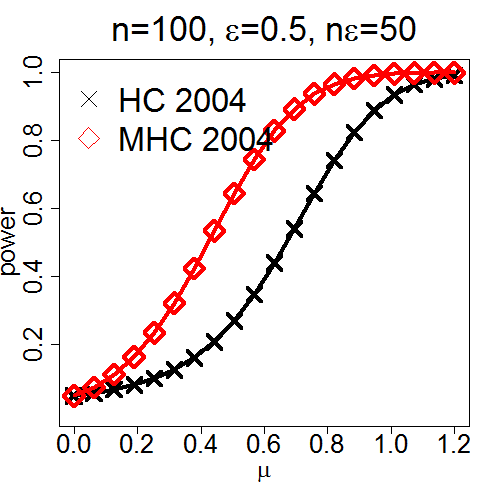}}\\
		\subfloat{\includegraphics[width=0.33\textwidth,height=1.8in]{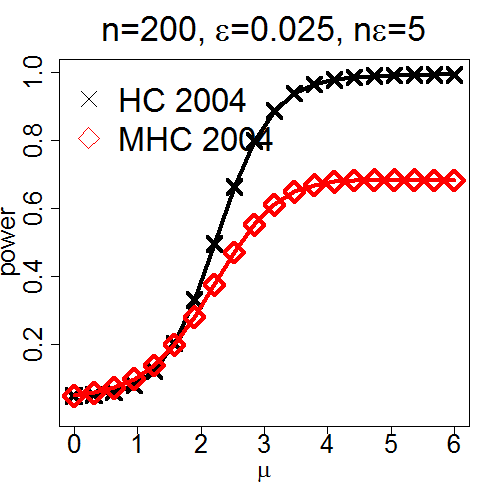}}
		\subfloat{\includegraphics[width=0.33\textwidth,height=1.8in]{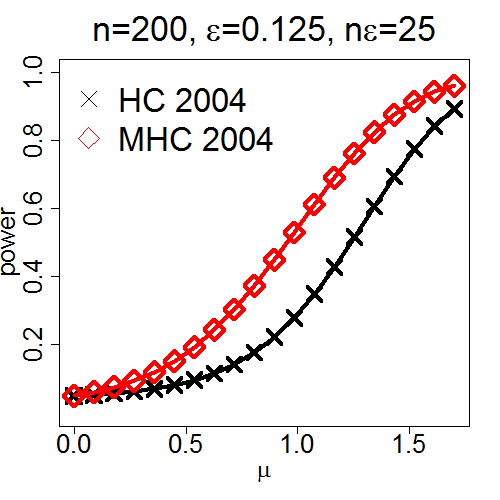}}
		\subfloat{\includegraphics[width=0.33\textwidth,height=1.8in]{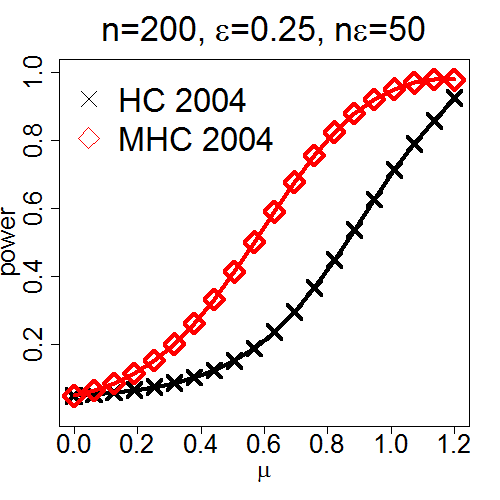}}
	\end{center}
\end{figure}

\section{A Genome-wide Association Study for Crohn's Disease\label{Sect: GWAS}}

According to the genetics of complex diseases, disease-associated markers usually have moderate to small effects\cite{Goldstein2009}. In genome-wise studies that tend to screen as many markers as possible, the number of true disease markers often account for a small proportion of the total candidates. Therefore, it is appealing to apply tests that are asymptotically optimal for detecting rare and weak signals. In this section, based on a logistic regression model, we applied optimal tests to a SNP-set association study at the gene level. Specifically, for the $i$th individual, $i = 1, ..., N$, let $y_i = 1$ (or $0$) indicate the case (or control), let $G_i = (G_{i1}, ..., G_{in})$ denote the genotype vector of $n$ SNPs in a given gene, and let $Z_{i}=(1, Z_{i1}, Z_{i2})$ contain the intercept and the first two principal components of the genotype data, which serve the purpose for controlling potential population structure \cite{Price2006}. The logistic regression model is  $logit(E(y_i | G_i, Z_i)) = G_i^{\prime} \beta + Z_i^{\prime} \gamma$, where $\beta$ and $\gamma$ are the coefficient parameters. A classic marginal-test statistic of the $j$th SNP, adjusted for the non-genetic measures, is \cite{mccullagh1989generalized, schaid2002score}
\[
U_j = \sum_{i=1}^N G_{ij}(y_i-\tilde{y}_i), \text{ }  j = 1, ..., n,
\]
where $\tilde{y}_i$ is the fitted outcome value under $H_0 : \beta=0$, i.e., none SNPs in the gene are associated. Under $H_0$, the vector of statistics $U=(U_1, ..., U_n) \rightsquigarrow N(0, \Sigma)$, where $\Sigma$ can be estimated by 
\[
\hat{\Sigma} = G^{\prime}WG - G^{\prime}WZ(Z^{\prime}WZ)^{-1}Z^{\prime}WG,
\]
where $G=(G_{ij})$ and $Z=(Z_{ij})$ are the corresponding design matrices, and $W=diag(\tilde{y}_i(1-\tilde{y}_i))$. After de-correlation we get test statistics $X = \hat{\Sigma}^{-\frac{1}{2}}U \rightsquigarrow N(0, I_{n \times n})$. Thus the i.i.d. condition of $H_0$ in (\ref{equ: generalHypo}) is reasonably satisfied, and our calculation methods can apply to obtain the p-value of a GOF statistic in (\ref{equ: GOFstat}), which measures the association of the given gene.  

We examined $HC^{2004}$, B-J, reverse B-J, and $HC^{2008}$ (again, they corresponding to the $\phi$-divergence statistics with $s=2, 1, 0, -1$, respectively). The GWAS data from NIDDK-IBDGC (National Institute of Diabetes, Digestive and Kidney Diseases - Inflammatory Bowel  Disease Genetics Consortium) contain 1,145 individuals from non-Jewish population (572 Crohn's disease cases and 573 controls) \cite{Duerr2006}. After typical quality control for genotype data, 308,330 somatic SNPs were grouped into 15,857 genes according to their physical locations. 

Figure \ref{fig:GWAS_Hexact} gives the QQ plots of the gene-association-indicating p-values calculated by Theorem \ref{Thm: exact}. The genomic inflation factors, i.e., the ratios of empirical median of -log(p-values) vs. the expected median under $H_0$, are all close to 1, evidencing that the genome-wide type I errors were well controlled. Among the four tests, the B-J seemed having higher power because it yielded more genes significantly above the red line of the $H_0$-expected p-values. Among the top ranked genes, {\it IL23R} and {\it CARD15} (also known as {\it NOD2}) are well-known Crohn's disease genes \cite{hugot2001association, ogura2001frameshift, Duerr2006}. Gene {\it NPTX2} was top ranked by both $HC^{2004}$ and B-J. It hasn't been reported to be directly associated, but it encodes a member of the family of neuronal petraxins, synaptic proteins that are related to C-reactive protein \cite{maglott2011entrez}, while C-reactive protein is an indicator for the activity level of Crohn's disease \cite{chamouard2006diagnostic}. Furthermore, {\it NPTX2} has an important paralog gene {\it APCS} (www.genecards.org), which is related to arthritis, a disease highly correlated with Crohn's disease \cite{trikudanathan2012diagnosis}. Gene {\it SLC44A4} is also related to the pathophysiology of Crohn's. Defects in this gene can cause sialidosis, a lysosomal storage disease \cite{maglott2011entrez} that results from a deficiency of the enzyme sialidase, the later is important for various cells to defend against infection \cite{james2011andrew}.  Gene {\it BMP2} was identified by B-J, reversed B-J, and $HC^{2008}$. This gene could also be relevant because it is associated with digestive phenotypes, especially colon cancer \cite{yuvaraj2012coli, slattery2012genetic}. Meanwhile, for those top ranked genes, further studies are needed to validate. 

\section{Discussion\label{Sect: Discu}}

This paper provides techniques to calculate the exact and approximated null and alternative distributions of a general family of GOF tests. Thus we can calculate both the p-value and the statistical power of these tests. These calculations are important for both practical applications of GOF tests and for performance comparisons under finite $n$. Comparing with the relevant literature, our methods are novel, accurate, general to broad statistic family and supremum domains. 

To calculate the exact distribution, the result in Theorem \ref{Thm: exact} brings down the computational complexity from $O(n^3)$ to $O(n^2)$ when comparing with corresponding literature methods. In the meanwhile, when $n$ and the search range $k_1-k_0$ are large, the calculation could suffer from the loss of significant digits. In this case, we could truncate the summation to the first 25 - 30 terms, which yields a fairly accurate result and saves computation time too. Moreover, we could also apply proper approximated calculations, for which the accuracy increases together with $n$.

\appendix

\section{Proofs}\label{app}

\subsection{Proofs of the Main Theorems}

\subsubsection{Proof of Theorem \ref{Thm: exact}}
For $k=k_1,...,1$, define 
\[
a_k = \int_{u_{k_1}}^1\frac{n!}{(n-k_1)!}(1-x_{k_1})^{n-k_1}\int_{u_{k_1-1}}^{x_{k_1}}...\int_{u_{k}}^{x_{k+1}}dx_{k}...dx_{k_1-1}dx_{k_1}.
\] 
Then obviously $a_{k_1}=\frac{n!}{(n-k_1+1)!}\bar{F}_{B(1, m)}(u_{k_1})$, and for $k\leq k_1-1$,
	\begin{align*}
		a_k 
		& = \int_{u_{k_1}}^1\frac{n!}{(n-k_1)!}(1-x_{k_1})^{n-k_1}\int_{u_{k_1-1}}^{x_{k_1}}...\int_{u_{k+1}}^{x_{k+2}}x_{k+1}dx_{k+1}...dx_{k_1-1}dx_{k_1}  - u_{k}a_{k+1}\\
		&=\int_{u_{k_1}}^1\frac{n!}{(n-k_1)!}(1-x_{k_1})^{n-k_1}\frac{x_{k_1}^{k_1-k}}{(k_1-k)!}dx_{k_1}-\sum_{j=1}^{k_1-k}\frac{u_{k+j-1}^j}{j!}a_{k+j}\\
		&=\frac{n!}{(n-k_1+k)!}\bar{F}_{B(k_1-k+1,m)}(u_{k_1}) - \sum_{j=1}^{k_1-k}\frac{u_{k+j-1}^j}{j!}a_{k+j}
	\end{align*}
	
Now by Lemma \ref{Lemma: jointBeta},
	\begin{align*}
		&P(S_n\leq b)=P\{U_{(k)}> u_k, k_0\leq k \leq k_1\}\\
		=& \int_{u_{k_1}}^1\frac{n!}{(n-k_1)!}(1-x_{k_1})^{n-k_1}\int_{u_{k_1-1}}^{x_{k_1}}...\int_{u_{k_0}}^{x_{k_0+1}}\frac{x_{k_0}^{k_0-1}}{(k_0-1)!}dx_{k_0}...dx_{k_1-1}dx_{k_1}\\
		=&\int_{u_{k_1}}^1\frac{n!}{(n-k_1)!}(1-x_{k_1})^{n-k_1}\int_{u_{k_1-1}}^{x_{k_1}}...\int_{u_{k_0+1}}^{x_{k_0+2}}\frac{x_{k_0+1}^{k_0}}{k_0!}dx_{k_0+1}...dx_{k_1-1}dx_{k_1} - \frac{u_{k_0}^{k_0}}{k_0!}a_{k_0+1}\\
		=&\int_{u_{k_1}}^1\frac{n!}{(n-k_1)!}(1-x_{k_1})^{n-k_1}\frac{x_{k_1}^{k_1-1}}{(k_1-1)!}dx_{k_1} -\sum_{i=k_0}^{{k_1}-1}\frac{u_{i}^i}{i!}a_{i+1}\\
		=&\bar{F}_{B(k_1,m)}(u_{k_1}) -\sum_{i=k_0}^{{k_1}-1}\frac{u_{i}^i}{i!}a_{i+1}.
	\end{align*}

\subsubsection{Proof of Theorem \ref{Thm: exacttruncate}}

The idea is to use total probability theorem.
	\begin{align*}
		P(S_{n,\mathcal{R}} \leq b) = &\sum_{i=1}^{k_1}\sum_{j=i+1}^{n+1} P\left(\{S_{n,\mathcal{R}} \leq b\}\cap \{\text{exactly } p_{(i)}...p_{(j-1)} \text{ fall in } [\alpha_0,\alpha_1]\}\right) \\
		= &\sum_{i=1}^{k_1}\sum_{j=i+1}^{n+1} P_{ij} 
	\end{align*}
Notice that $1\leq i\leq n$, $i+1\leq j \leq n+1$, we have
	\begin{align*}
		P_{ij}=P(U_{(i-1)}<\beta_0,U_{(i)}\geq u_i,...,U_{(j-1)}\geq u_{j-1},U_{(j-1)}\leq \beta_1,U_{(j)}>\beta_1),
	\end{align*}
where the joint density of $U_{(i-1)},...,U_{(j)}$ is 
	\begin{align*}
		f(x_{i-1},...,x_{j})=\frac{n!}{(i-2)!(n-j)!}x_{i-1}^{i-2}(1-x_j)^{n-j},\quad{}0\leq x_{i-1}\leq...\leq x_{j}\leq 1.
	\end{align*}
Then 
	\begin{align*}
		P_{ij}&=n!\int_{\beta_1}^{1}\frac{(1-x_j)^{n-j}}{(n-j)!}\int_{u_{j-1}}^{\beta_1}\cdots \int_{u_{i}}^{x_{i+1}}\int_{0}^{\beta_0}\frac{x_{i-1}^{i-2}}{(i-2)!}dx_{i-1}dx_i...dx_{j-1}dx_j\\
		&=n!\int_{\beta_1}^{1}\frac{(1-x_j)^{n-j}}{(n-j)!}dx_j\int_{u_{j-1}}^{\beta_1}\cdots \int_{u_{i}}^{x_{i+1}}dx_i...dx_{j-1}\int_{0}^{\beta_0}\frac{x_{i-1}^{i-2}}{(i-2)!}dx_{i-1}\\
		&=c_{ij}n!\int_{u_{j-1}}^{\beta_1}\cdots \int_{u_{i}}^{x_{i+1}}dx_i...dx_{j-1}.
	\end{align*}
Direct calculation similar to the proof of Theorem \ref{Thm: exact} gives the final result.

\subsubsection{Proof of Theorem \ref{Thm: exactgeneral}}

Following the idea in the proof of Theorem \ref{Thm: exacttruncate}, 
	\begin{align*}
		P(S_{n,\mathcal{R}} \leq b) = &\sum_{i=1}^{k_1}\sum_{j=\tilde{i}+1}^{n+1} P\left(\{S_{n,\mathcal{R}} \leq b\}\cap \{\text{exactly } p_{(i)}...p_{(j-1)} \text{ fall in } [\alpha_0,\alpha_1]\}\right) \\
		= &\sum_{i=1}^{k_1}\sum_{j=\tilde{i}+1}^{n+1} P_{ij}.
	\end{align*}
For each feasible pair of $(i,j)$, direct calculation shows $P_{ij}$ can be concisely written as
	\begin{align*}
		P_{ij} &= c_{ij}n! \int_{u_{\tilde{j}-1}}^{\beta_1}\frac{(\beta_1-x_{\tilde{j}-1})^{j-\tilde{j}}}{(j-\tilde{j})!}\int_{u_{\tilde{j}-2}}^{x_{\tilde{j}-1}}\cdots \int_{u_{\tilde{i}}}^{x_{\tilde{i}+1}}\frac{(x_{\tilde{i}}-\tilde{\beta}_0)^{\tilde{i}-i}}{(\tilde{i}-i)!}dx_{\tilde{i}}...dx_{\tilde{j}-2}dx_{\tilde{j}-1}\\
		&= c_{ij}\left(n!\frac{(\beta_1-\tilde{\beta}_0)^{j-i}}{(j-i)!}\bar{F}_{B(\tilde{j}-i,j-\tilde{j}+1)}(\frac{u_{\tilde{j}-1}-\tilde{\beta}_0}{\beta_1-\tilde{\beta}_0}) - \sum_{k=\tilde{i}}^{\tilde{j}-1}\frac{(u_k-\tilde{\beta}_0)^{k-i+1}}{(k-i+1)!}a_j(k+1)\right),
	\end{align*}
where 
	\begin{align*}
		a_{j}(k) &= n! \int_{u_{\tilde{j}-1}}^{\beta_1}\frac{(\beta_1-x_{\tilde{j}-1})^{j-\tilde{j}}}{(j-\tilde{j})!}\int_{u_{\tilde{j}-2}}^{x_{\tilde{j}-1}}\cdots \int_{u_{k}}^{x_{k+1}}dx_{k}...dx_{\tilde{j}-2}dx_{\tilde{j}-1}\\
		&= n!\frac{\beta_1^{j-k}}{(j-k)!}\bar{F}_{B(\tilde{j}-k,j-\tilde{j}+1)}(\frac{u_{\tilde{j}-1}}{\beta_1}) - \sum_{l=1}^{\tilde{j}-k-1}\frac{u_{k+l-1}^{l}}{l!}a_j(k+1).
	\end{align*}

\subsubsection{Proof of Theorem \ref{Thm: general}}

      The main idea of the proofs of Theorem \ref{Thm: general} and \ref{Thm: linearCase} is as follows. Note that $U_{(i)}:=D(p_{(i)})$ defined in (\ref{equ: P(Sn<b)}) follow the same distribution of $\frac{\Gamma_i}{\Gamma_{n+1}}$, where $\Gamma_i = \sum_{j=1}^i \varepsilon_j$, $\varepsilon_j\overset{\text{i.i.d.}}{\sim} Exp(1)$, so that $\Gamma_i \sim Gamma(i, 1)$. Thus we can approximate 
\begin{align*}
	&P\{U_{(k)}> D(g(\frac{k}{n}, b)), \text{for all }k_0\leq k\leq k_1\}. \\
     \approx &P\{\Gamma_{(k)}> (n+1)D(g(\frac{k}{n}, b)), \text{for all }k_0\leq k\leq k_1\}.
\end{align*}
We take advantage of the joint density of $(\Gamma_{k_0},...,\Gamma_{k_1})$, which is given by Lemma \ref{Lemma: jointGamma}, while $\Gamma_{n+1}$ can be approximated by $n+1$ when $n$ is reasonably large.

Accordingly, we apply similar calculation as the proof of Theorem \ref{Thm: exact} except using \ref{Lemma: GammaApprox} and Lemma \ref{Lemma: GammaSurvivalRecursion} instead. 
 
\subsubsection{Proof of Theorem \ref{Thm: linearCase}}
Note that
	\begin{align*}
		  &P(\Gamma_k>a+\lambda k, 1\leq k \leq {k_1})\\
		=&\int_{a+\lambda {k_1}}^{+\infty}e^{-z_{k_1}}\left( \int_{a+\lambda ({k_1}-1)}^{z_{{k_1}}}...\int_{a+\lambda }^{z_{2}}dz_1...dz_{{k_1}-1}\right )dz_{k_1}\\
		  &\text{By lemma \ref{Lemma: Abel-Goncharov}, }\\
		=&\int_{a+\lambda {k_1}}^{+\infty}e^{-z_{k_1}}\frac{\left(z_{k_1}-a-({k_1}-1)\lambda \right)(z_{k_1}-a)^{{k_1}-2}}{({k_1}-1)!}dz_{k_1}\\
		=&\int_{a+\lambda {k_1}}^{+\infty}\left(e^{-z_{k_1}}\frac{(z_{k_1}-a)^{{k_1}-1}}{({k_1}-1)!}-e^{-z_{k_1}}\frac{({k_1}-1)\lambda(z_{k_1}-a)^{{k_1}-2}}{({k_1}-1)!}\right)dz_{k_1}\\
		  &\text{Let $x=z_{k_1}-a$, }\\
		=&e^{-a}\left(\int_{\lambda {k_1}}^{+\infty}e^{-x}\frac{x^{{k_1}-1}}{({k_1}-1)!}dx - \lambda \int_{\lambda k_1}^{+\infty}e^{-x}\frac{x^{{k_1}-2}}{({k_1}-2)!}dx\right)\\
		=&e^{-a}\left[1-P(\Gamma_{k_1}\leq \lambda {k_1})-\lambda \left( 1-P(\Gamma_{{k_1}-1}\leq \lambda {k_1})\right)\right]\\
		=&e^{-a}(1-\lambda +\lambda F_{\Gamma({k_1}-1)}(\lambda {k_1})-F_{\Gamma(k_1)}(\lambda {k_1}))\\
             =&e^{-a}(1-\lambda +h_{k_1}(\lambda)).
	\end{align*}
Thus Theorem \ref{Thm: linearCase} is proved by combining this equation and Lemma \ref{Lemma: GammaApprox}.  
%\subsubsection{Proof of Corollary \ref{Corol: KS}}
%When $b\leq \frac{1}{n}$
%       \begin{align*}
%               P(KS^+\leq b)=&P(p_{(k)}\geq \frac{k}{n}-b,1\leq k \leq {k_1}) \\
%             =&(1+o(1))P(\Gamma_{(k)}\geq -(n+1)b+\frac{(n+1)}{n}k,1\leq k \leq {k_1})
%       \end{align*}
%Theorem \ref{Thm: linearCase} yields the desired result.
\newline
\newline \indent The idea of the proof of Theorem \ref{Thm: CloseLinearCase} is motivated by \cite{Li2014higher}. Instead of directly considering the distribution function, we look at the right-tail probability which can be decomposed into the union of disjoint sets. 

\subsubsection{Proof of Theorem \ref{Thm: CloseLinearCase} and Corollary \ref{Corol: HC}}

Let event $A_{n,k}$ be defined as in Lemma \ref{Lemma: Ank}. They are disjoint and $\{S_n\geq b\}=\bigcup_{k=k_0}^{k_1}A_{n,k}$. In this proof we mainly focus on approximating $P(A_{n,k})$
\newline Let $d_{k}=(n+1)D(g(\frac{k}{n}, b))$, $d_{k}^\prime=(n+1)\frac{d D(g(\frac{k}{n}, b))}{d x}$, $d_{k,max}^\prime=(n+1)\frac{d_{k_1}-d_k}{k_1-k}$. Notice that $D(g(x, b))$ is convex in $x$, so $d_{k+j}>d_k+\frac{d_k^{'}}{n}j$. From Lemma \ref{Lemma: GammaSurvivalRecursion} and Lemma \ref{Lemma: Ank}, we have for $1\leq k\leq k_1-1$,
\begin{align*} 
	P(A_{n,k})&=\frac{d_k^k}{k!}(\Gamma_j>d_{k+j}, 1\leq j \leq k_1 -k)\\
	&\leq \frac{d_k^k}{k!}P(\Gamma_j>d_k+\frac{d_k^{'}}{n}j, 1\leq j \leq k_1 -k)\\
	&=\frac{d_k^k}{k!}e^{-d_k}(1-\frac{d_k^{'}}{n}+h_{k_1-k}(\frac{d_k^{'}}{n}))\\
	&=(1-\frac{d_k^{'}}{n}+h_{k_1-k}(\frac{d_k^{'}}{n}))f_{P(d_k)}(k).
\end{align*}

Next we need to find the lower bound for $P(A_{n,k})$. 
\newline We first consider $k\geq k_1-\sqrt{n}$, here $\sqrt{n}$ is chosen for covenience and the proof works for any $n^{\gamma}$, $0<\gamma<1$. 
Note that, due to the convexity, $d_{k+j}\leq d_k+\frac{d_{k+\sqrt{n}}^{'}}{n}j$, $1\leq j\leq \sqrt{n}$
\begin{align*}
	P(A_{n,k})&\geq \frac{d_k^k}{k!}P(\Gamma_j\geq d_k+\frac{d_{k+\sqrt{n}}^\prime}{n}j, 1\leq j\leq \sqrt{n})\\
	&\geq \frac{d_k^k}{k!}e^{-d_k}(1-\frac{d_{k+\sqrt{n}}^\prime}{n}+h_{\sqrt{n}}(\frac{d_{k+\sqrt{n}}^\prime}{n})\\
	%&=(1+o(1))\frac{d_k^k}{k!}e^{-d_k}(1-\frac{d_k^\prime}{n}+h_k(\frac{d_k^\prime}{n}))\\
	&=(1+o(1))(1-\frac{d_k^{'}}{n}+h_{\sqrt{n}}(\frac{d_k^{'}}{n}))f_{P(d_k)}(k).
\end{align*}
The last equation is due to Theorem \ref{Thm: linearCase} and the continuity of $D(g(x, b))$. We can see $P(A_{n,k})\to (1-\frac{d_k^{'}}{n}+h_{\sqrt{n}}(\frac{d_k^{'}}{n}))f_{P(d_k)}(k)$ uniformly in $k$, $k\geq k_1-\sqrt{n}$.
\newline
\newline We then consider $k\leq k_1-\sqrt{n}$. In this case, the proof is slightly more complicated than the first case, however the idea is similar. 
\newline Let $p_{n,k}(y)=P(\Gamma_j>d_{k+j}-d_k+y, 1\leq j \leq k_1 -k)$, $0\leq y\leq d_k$, and $f_{\Gamma_{k}}$ be the density function of $\Gamma_{k}$, then $P(A_{n,k})=\int_0^{d_k}p_{n,k}(y)f_{\Gamma_{k}}(d_k-y)dy$.
\newline Similar to the proof in the first case, 
\begin{align*}
	p_{n,k}(y)&\geq P(\Gamma_j\geq y+\frac{d_{k+\sqrt{n}}^\prime}{n}, 1\leq j\leq \sqrt{n})-P(\bigcup_{j=\sqrt{n}}^{k_1-k}\{\Gamma_j\leq y+d_{k+j}-d_k\})\\
	&\geq e^{-y}(1-\frac{d_{k+\sqrt{n}}^\prime}{n}+h_{\sqrt{n}}(\frac{d_{k+\sqrt{n}}^\prime}{n}))-\sum_{j=\sqrt{n}}^{k_1-k} P(\frac{\Gamma_j}{j}\leq \frac{y}{j}+\frac{d_{k,max}^\prime}{n})\\
	&=(1+o(1))e^{-y}(1-\frac{d_k^\prime}{n}+h_{\sqrt{n}}(\frac{d_k^\prime}{n}))-\sum_{j=\sqrt{n}}^{k_1-k} P(\frac{\Gamma_j}{j}\leq \frac{y}{j}+\frac{d_{k,max}^\prime}{n}).
\end{align*}

To prove the residual uniformly in $k$ converges to 0, we need to apply Lemma \ref{Lemma: gammaTail}. For $y$ of some lower order, say $O(\log n)$
\begin{align*}
	\frac{y}{j}+\frac{d_{k,max}^\prime}{n}&\leq \frac{\log n}{\sqrt{n}}+\frac{d_{k,max}^\prime}{n}\\
	&<\frac{\log n}{\sqrt{n}}+\frac{n+1}{n}\sup_{\frac{k_0}{n}\leq x\leq \frac{k_1}{n}}\{\frac{d D(g(x, b))}{d x}\}\\
	&\leq \delta<1\text{\quad when $n$ is large}.
\end{align*}

Then, by Lemma \ref{Lemma: gammaTail}, we have the residual
\begin{align*}
	\text{res}_k\leq \sum_{j=\sqrt{n}}^{k_1-k}e^{-jI(\delta)}\leq \frac{e^{-I(\delta)\sqrt{n}}}{1-e^{-I(\delta)}}\to 0 \text{ exponentially uniformly in k}.
\end{align*}
From
\begin{align*}
	(1+o(1))e^{-y}(1-\frac{d_k^\prime}{n}+h_{\sqrt{n}}(\frac{d_k^\prime}{n}))-res_k\leq p_{n,k}\leq e^{-y}(1-\frac{d_k^\prime}{n}+h_{\sqrt{n}}(\frac{d_k^\prime}{n})),
\end{align*}
we can conclude that 
\begin{align*}
	&p_{n,k}(y)\to e^{-y}(1-\frac{d_k^\prime}{n}+h_{\sqrt{n}}(\frac{d_k^\prime}{n}))\text{ uniformly in k}, \\
	&\text{ with error bound } res_k \leq \frac{e^{-I(\delta)\sqrt{n}}}{1-e^{-I(\delta)}}.
\end{align*}

Let $\alpha$ denote $\log n$ for simplicity, $\alpha>1$, then, for some constant $c>1$,
\begin{equation}
\label{equ: ItoIhat}
	\begin{aligned}
		& \hat{I}_{n,k} =\int_0^{\min\{d_k,c\alpha\}}e^{-y}(1-\frac{d_k^\prime}{n}+h_{\sqrt{n}}(\frac{d_k^\prime}{n}))f_{\Gamma_{k}}(d_k-y)dy\\
		\to&I_{n,k}=\int_0^{\min\{d_k,c\alpha\}}p_{n,k}(y)f_{\Gamma_{k}}(d_k-y)dy\text{\quad uniformly with error bound } \leq \frac{ne^{-I(\delta)\sqrt{n}}}{1-e^{-I(\delta)}}.
	\end{aligned}
\end{equation}
When $d_k>c\alpha$, $d_k-y<d_k<k-1$, $f_{\Gamma_k}(d_k-y)$ is decreasing in $y$
\begin{align*}
	\frac{P(A_{n,k})-I_{n,k}}{I_{n,k}}&\leq \frac{(d_k-c\alpha) f_{\Gamma_{k}}(d_k-c\alpha)}{\alpha f_{\Gamma_{k}}(d_k-\alpha)}\\
	&=\frac{(d_k-c\alpha)e^{c\alpha-d_k}(d_k-c\alpha)^{k-1}}{\alpha e^{\alpha-d_k}(d_k-\alpha)^{k-1}}\\
	&=\frac{(d_k-c\alpha)}{\alpha}e^{(c-1)\alpha}(1-\frac{(c-1)\alpha}{d_k-\alpha})^{k-1}\\
	&=\frac{(d_k-c\alpha)}{\alpha}e^{(c-1)\alpha}e^{(k-1)\log (1-\frac{(c-1)\alpha}{d_k-\alpha})}\\
	&\leq\frac{d_k}{\alpha}e^{\alpha(c-1)}e^{-(k-1)(\frac{(c-1)\alpha}{d_k-\alpha})}\\
	&\leq\frac{d_k}{\alpha}e^{\alpha(c-1)(1-\frac{k-1}{d_k})}\\
	&\text{Let $x=\frac{k}{n}$, $y=D(g(x, b))$, then}\\
	&=\frac{(n+1)y}{\alpha}n^{-(c-1)(\frac{nx-1}{(n+1)y}-1)}.
\end{align*}

Notice that $\frac{nx-1}{(n+1)y}>1$, we need $(c-1)(\frac{nx-1}{(n+1)y}-1)>1$ for all $x$, that is $c=\sup_{\frac{2}{n}\leq x\leq \frac{k_1}{n}} \frac{nx-1}{nx-1-(n+1)D(g(x, b))}>1$. Therefore $\frac{P(A_{n,k})-I_{n,k}}{I_{n,k}}\leq \frac{y}{\alpha}n^{-\gamma}$, $\gamma>0$. For $k=1$, the result follows $d_1\to 0$,
\begin{align}
\label{equ: ItoP}
	I_{n,k}\to P(A_{n,k}) \text{ uniformly in k as $n\to \infty$}.
\end{align}

Define $\hat{P}(A_{n,k})=\int_0^{d_k}e^{-y}(1-\frac{d_k^\prime}{n}+h_{\sqrt{n}}(\frac{d_k^\prime}{n}))f_{\Gamma_{k}}(d_k-y)dy$, 
\begin{align*}
	\frac{\hat{P}(A_{n,k})-\hat{I}_{n,k}}{\hat{P}(A_{n,k})}&=1-\frac{\int_0^{\min\{d_k,c\alpha\}}\frac{(d_k-y)^{k-1}}{(k-1)!}dy}{\int_0^{d_k}\frac{(d_k-y)^{k-1}}{(k-1)!}dy}\\
	&=1-(\frac{d_k^k-(d_k-c\alpha)^k}{d_k^k})\\
	&=\left(1-\frac{c\alpha}{d_k}\right)^k\\
	&\leq \exp\left(-c\frac{k}{d_k}\log n\right).
\end{align*}
Since $c,\frac{k}{d_k}>1$, we have 
\begin{align}
\label{equ: IhattoPhat}
	\hat{I}_{n,k}\to \hat{P}(A_{n,k}) \text{ uniformly in k as $n\to \infty$}.
\end{align}

Finally, notice that $e^{-y}f_{\Gamma_{k}}(d_k-y)=e^{-d_k}\frac{(d_k-y)^{k-1}}{(k-1)!}$, then
\begin{equation}
\label{equ: Phat}
	\begin{aligned}
		\hat{P}(A_{n,k})&=(1-\frac{d_k^\prime}{n}+h_{\sqrt{n}}(\frac{d_k^\prime}{n}))e^{-d_k}\int_0^{d_k}\frac{(d_k-y)^{k-1}}{(k-1)!}dy\\
		&=(1-\frac{d_k^{'}}{n}+h_{\sqrt{n}}(\frac{d_k^{'}}{n}))f_{P(d_k)}(k).
	\end{aligned}
\end{equation}

Combine lemma \ref{Lemma: GammaApprox} and equations (\ref{equ: ItoIhat}), (\ref{equ: ItoP}), (\ref{equ: IhattoPhat}), and (\ref{equ: Phat}), we have
\begin{align*}
	P(S_n\geq b)=&(1+o(1))P(\Gamma_k\leq d_k, \text{ for some }k_0\leq k \leq k_1)\\
	=&(1+o(1))\sum_{k=k_0}^{k_1}\hat{P}(A_{n,k})\\
	=&(1+o(1))\sum_{k=k_0}^{k1}(1-\frac{d_k^\prime}{n}+h_{k^*}(\frac{d_k^\prime}{n}))f_{P(d_k)}(k),
\end{align*}
where $k^*=\min\{k_1-k,\sqrt{n}\}$.

\bigskip
To proof Corollary \ref{Corol: HC}, by Theorem \ref{Thm: CloseLinearCase}, we have
\[
	P(HC^{2004}\geq b)=(1+o(1))\sum_{k=k_0}^{k1}\left(1-\frac{n+1}{n}g^\prime(\frac{k}{n}, b_0)+h_{k^*}(\frac{n+1}{n}g^\prime(\frac{k}{n}, b_0))\right)f_{P((n+1)g(\frac{k}{n}, b_0))}\left(k\right).
\]
%Note that by Lemma \ref{Lemma: gammaTail}, $h_{k^*}(\frac{n+1}{n}g_2^\prime(\frac{k}{n}, b_0))\to 0$, $n\to \infty$ exponentially. We have
%\begin{align*}
%	P(HC^{2004}\geq b)=(1+o(1)) \sum_{k=k_0}^{k1}\left(1-\frac{n+1}{n}g_2^\prime(\frac{k}{n}, b_0)\right)F_p\left(k,(n+1)g(\frac{k}{n}, b_0)\right).
%\end{align*}

\bigskip
\subsection{Fundmental Lemmas and Proofs}
\begin{lemma} 
\label{Lemma: jointBeta}
	\emph{}
	Let $U_{(i)}$ be the $i^{th}$ order statistic of \text{i.i.d.} samples from Uniform(0,1),  $1\leq k_0<k_1\leq n$.
      Then the joint density of $(U_{(k_0)},..., U_{(k_1)})$ is
	\begin{align*}
		&f_{(U_{(k_0)},..., U_{(k_1)})}(z_{k_0},...,z_{k_1})=\frac{n!}{(n-k_1)!(k_0-1)!}z_{k_0}^{k_0-1}(1-z_{k_1})^{n-k_1}, 0<z_{k_0}<...<z_{k_1}<1.
	\end{align*}
\end{lemma}

\begin{proof}
	Standard result from formula (2.2.2) of the book Order Statistics\cite{david2003order}. 
\end{proof}

\begin{lemma} 
\label{Lemma: jointGamma}
	\emph{}
	Let $\varepsilon_i$ be \text{i.i.d.} exponential random variables with parameter 1. $\Gamma_k=\sum_{i=1}^{k}\varepsilon_i$, $1\leq k_0 <  k_1\leq n$. 
      Then the joint density of $(\Gamma_{k_0},..., \Gamma_{k_1})$
	\begin{align*}
		&f(z_{k_0},...,z_{k_1})=\frac{z_{k_0}^{k_0-1}}{(k_0-1)!}e^{{-z_{k_1}}}, 0<z_{k_0}<...<z_{k_1}<\infty.
	\end{align*}

	Specially, when $k_0=1$, the joint density is
	\begin{align*}
		&f(z_{1},...,z_{k_1})=e^{{-z_{k_1}}}, 0<z_{1}<...<z_{k_1}<\infty.
	\end{align*}
\end{lemma}

\begin{proof}
	The proof follows the deduction process shown in Theorem 1.1 by Mathai and Moschopoulos \cite{mathai1991multivariate}. Note that their theorem only considers the case of $k_0=1$, but the deduction idea can be applied to general $k_0$. 
\end{proof}

\begin{lemma}
\label{Lemma: gammaTail}
	Let $c<1$, $I(c)=c-1-\log(c)>0$, $\Gamma_n=\sum_{i=1}^{n}\epsilon_i$, where $\epsilon_i$ are i.i.d exponential distributed with parameter 1, then
	\begin{align*}
		P(\frac{\Gamma_n}{n}\leq c)\leq e^{-nI(c)}.
	\end{align*}
\end{lemma}

\begin{proof}
	Let $\epsilon_i$ be $\exp(1)$ distributed. The log moment generating function of $-\epsilon_i$ is $\Lambda(t)=-\log (1+t)$, $t>-1$. Then the convex rate function is
	\begin{align*}
		\Lambda^*(x)
		&=\sup_{t>-1} tx-\Lambda(t)\\
		&=-x-1-\log(-x).
	\end{align*}
	By Cramer's theorem, let $x=-c$, we have desired result.
\end{proof}

\begin{lemma}
\label{Lemma: Abel-Goncharov}
	\emph{(Abel-Goncharov Polynomial)}
	\begin{align*}
		\int_{a+\lambda ({k_1}-1)}^{z_{{k_1}}}...\int_{a+\lambda }^{z_{2}}dz_1...dz_{{k_1}-1}=\frac{\left(z_{k_1}-a-\lambda ({k_1}-1)\right)(z_{k_1}-a)^{{k_1}-2}}{({k_1}-1)!}.
	\end{align*}
\end{lemma}
\begin{proof}
	We will prove by induction. 
	\newline When ${k_1}=2$, it's easily shown that
	\begin{align*}
		\int_{a+\lambda }^{z_{2}}dz_1=z_2-a-\lambda =\frac{\left(z_2-a-\lambda (2-1)\right)(z_2-a)^{2-2}}{(2-1)!}.
	\end{align*}
	Assume it's true for ${k_1}=k$,
	\begin{align*}
		\int_{a+\lambda (k-1)}^{z_{k}}...\int_{a+\lambda }^{z_{2}}dz_1...dz_{k-1}=\frac{\left(z_k-a-\lambda (k-1) \right)(z_k-a)^{k-2}}{(k-1)!}.
	\end{align*}
	Then for ${k_1}=k+1$,
	\begin{align*}
		  &\int_{a+\lambda k}^{z_{k+1}}\int_{a+\lambda (k-1)}^{z_{k}}...\int_{a+\lambda }^{z_{2}}dz_1...dz_{k-1}dz_k\\
		=&\int_{a+\lambda k}^{z_{k+1}}\frac{\left(z_k-a-\lambda (k-1)\right)(z_k-a)^{k-2}}{(k-1)!}dz_k\\
		=&\int_{a+\lambda k}^{z_{k+1}}\frac{(z_k-a)^{k-1}}{(k-1)!}d_k-\int_{a+\lambda k}^{z_{k+1}}\frac{\lambda (z_k-a)^{k-2}}{(k-2)!}d_k\\
		=&\frac{(z_{k+1}-a)^k}{k!}-\frac{(\lambda k)^k}{k!}-\frac{\lambda (z_{k+1}-a)^{k-1}}{(k-1)!}+\frac{\lambda (\lambda k)^{k-1}}{(k-1)!}\\
		=&\frac{(z_{k+1}-a)^{k-1}(z_{k+1}-a-\lambda k)}{k!}.
	\end{align*}
	This finished the proof.
\end{proof}

\begin{lemma} 
\label{Lemma: GammaApprox}
	Let $c=(c_{k_0},...,c_{k_1}), 1\leq{k_0} < {k_1}\leq n$ be a finite sequence of increasing numbers. Then given $\epsilon>0$, as $n\to \infty$
	\begin{align*}
		P(\frac{\Gamma_j}{\Gamma_{n+1}}>c_j, {k_0}\leq j \leq {k_1})=(1+O(\epsilon)) P(\Gamma_j>(n+1)c_j, {k_0}\leq j \leq {k_1}).
	\end{align*}
\end{lemma}
\begin{proof}
	Given $\epsilon>0$, let $L(c;n,\epsilon)$ and $R(c;n,\epsilon)$ be the lower bound and upper bound for $P(\frac{\Gamma_j}{\Gamma_{n+1}}>c_j, {k_0}\leq j \leq {k_1})$,
	\begin{align*}
		L(c;n,\epsilon)=&P(\Gamma_j>(1+\epsilon)(n+1)c_j, {k_0}\leq j \leq {k_1})\\
		-&P(\frac{\Gamma_{n+1}}{n+1}\in [1-\epsilon,1+\epsilon]^c), \\
		R(c;n,\epsilon)=&P(\Gamma_j>(1-\epsilon)(n+1)c_j, {k_0}\leq j \leq {k_1})\\
		+&P(\frac{\Gamma_{n+1}}{n+1}\in [1-\epsilon,1+\epsilon]^c).
	\end{align*}
	By Lemma \ref{Lemma: gammaTail},
	\begin{align*}
		P(\frac{\Gamma_{n+1}}{n+1}\in [1-\epsilon,1+\epsilon]^c)\leq \frac{1}{\epsilon^2}e^{-n\Lambda(1-\epsilon)}\to 0 \text{ exponentially.}
	\end{align*}
	Because of the continuity of the multivariate $gamma$ distribtution, 
	\begin{align*}
		L(c;n,\epsilon)\to &(1+O(\epsilon))P(\Gamma_j>(n+1)c_j, {k_0}\leq j \leq {k_1}),\\
		R(c;n,\epsilon)\to &(1+O(\epsilon))P(\Gamma_j>(n+1)c_j, {k_0}\leq j \leq {k_1}).
	\end{align*}
\end{proof}

\begin{lemma}
%\label{survival}
\label{Lemma: GammaSurvivalRecursion}
\emph{}
	Let $(d_1,...,d_{k_1})$ be a sequence of nondecreasing and nonnegative numbers. 
	$Q_{j}(d_{j+1},...,d_{{k_1}})=P\{\Gamma_k\geq d_{k+j}, 1\leq k\leq {k_1}-j\}$, $0\leq j \leq {k_1}-1$. 
	$\bar{F}_{\Gamma_k}(x)$ is the survival function of Gamma distribution with shape parameter $k$ and scale parameter $1$.  Then for $l=2,3,...,k_1$
	\begin{align*}
		Q_{k_1-l} = \bar{F}_{\Gamma_{l}}(d_{k_1}) - \sum_{j=1}^{l-1}\frac{d_{k_1-l+j}^j}{j!}Q_{k_1-l+j}.
		%&= \sum_{j=0}^{m-i}\frac{d_{m}^j}{j!}e^{-d_{m}}-\sum_{j=1}^{m-i}\frac{a_{j+i-1}^j}{j!}Q_{j+i-1}
	\end{align*}
      with $Q_{k_1-1}=\bar{F}_{\Gamma_{1}}(d_{k_1})$, and, for $k_0\geq 1$, the joint survival probability
      \[
            P\{\Gamma_k\geq d_{k}, k_0\leq k\leq {k_1}\}=\bar{F}_{\Gamma_{k_1}}(d_{k_1}) -\sum_{j=k_0}^{{k_1}-1}\frac{d_{j}^j}{j!}Q_{j}.
      \]
\end{lemma}
\begin{proof}
	\begin{align*}
		Q_0
		&=P\{\Gamma_k\geq d_k, 1\leq k\leq k_1\}\\
		&=\int_{d_{k_1}}^{+\infty}e^{-z_{k_1}}\int_{d_{{k_1}-1}}^{z_{{k_1}}}...\int_{d_{1}}^{z_{2}}dz_{1}...dz_{k_1-1}dz_{k_1}\\
		&=\int_{d_{k_1}}^{+\infty}e^{-z_{k_1}}\int_{d_{{k_1}-1}}^{z_{{k_1}}}...\int_{d_{2}}^{z_{3}}(z_2-d_1)dz_2...dz_{{k_1}-1}dz_{k_1}\\
		&=\int_{d_{k_1}}^{+\infty}e^{-z_{k_1}}\int_{d_{{k_1}-1}}^{z_{{k_1}}}...\int_{d_{2}}^{z_{3}}z_2dz_2...dz_{{k_1}-1}dz_{k_1}-d_1Q_1\\
		&=\int_{d_{k_1}}^{+\infty}e^{-z_{k_1}}\int_{d_{{k_1}-1}}^{z_{{k_1}}}...\int_{d_{k_0}}^{z_{k_0+1}}\frac{z_{k_0}^{k_0-1}}{(k_0-1)!}dz_{k_0}...dz_{{k_1}-1}dz_{k_1}-\sum_{j=1}^{{k_0}-1}\frac{d_{j}^j}{j!}Q_{j}\\
		&=\int_{d_{k_1}}^{+\infty}e^{-z_{k_1}}\frac{z_{k_1}^{{k_1}-1}}{({k_1}-1)!}dz_{k_1}-\sum_{j=1}^{{k_1}-1}\frac{d_{j}^j}{j!}Q_{j}\\
		&= \bar{F}_{\Gamma_{k_1}}(d_{k_1}) -\sum_{j=1}^{{k_1}-1}\frac{d_{j}^j}{j!}Q_{j}.
	\end{align*}
	The rest of recursive formulae can be derived similarly and from the last third equation. Also notice that $\{\Gamma_k\geq d_k, k_0\leq k\leq k_1\}=\{\Gamma_k\geq d_k, 1\leq k\leq k_1\}\bigcup \{\text{some }\Gamma_k < d_k, 1\leq k \leq k_0-1\text{ and }\Gamma_k\geq d_k, k_0\leq k\leq k_1\}$, by Lemma \ref{Lemma: Ank} we have
      \begin{align*}
             P\{\Gamma_k\geq d_{k}, k_0\leq k\leq {k_1}\}&=Q_0+\sum_{j=1}^{{k_0}-1}\frac{d_{j}^j}{j!}Q_{j}\\
             &=\bar{F}_{\Gamma_{k_1}}(d_{k_1}) -\sum_{j=k_0}^{{k_1}-1}\frac{d_{j}^j}{j!}Q_{j}.
      \end{align*}
\end{proof}

\begin{lemma}
\label{Lemma: Ank}
	Let $(d_1,...,d_{k_1})$ be a sequence of increasing numbers,  $A_{n,k}=\{\Gamma_k \leq d_k, \Gamma_{k+l} > d_{k+l}, 1\leq l\leq k_1-k\}$, $1\leq k \leq k_1-1$. Then
	\begin{align*}
		P(A_{n,k})
		&=\frac{d_{k}^{k}}{k!}Q_k.
	\end{align*}
\end{lemma}

\begin{proof}
	\begin{align*}
		P({A_{n,k}})
		&=P\{\Gamma_k \leq d_k, \Gamma_{k+l} > d_{k+l} \text{ for all }1\leq l\leq k_1-k\}\\
		&=\int_{d_{k_1}}^{+\infty}e^{{-z_{k_1}}}\int_{d_{k_1-1}}^{z_{k_1}}...
		    \int_{d_{k+1}}^{z_{k+1}}\int_{0}^{d_{k}}\frac{z_{k}^{k-1}}{(k-1)!}dz_kdz_{k+1}...dz_{k_1-1}dz_{k_1}\\
		&=\int_{d_{k_1}}^{+\infty}e^{{-z_{k_1}}}\int_{d_{k_1-1}}^{z_{k_1}}...\int_{d_{k+1}}^{z_{k+1}}\frac{d_{k}^{k}}{k!}dz_{k+1}...dz_{k_1-1}dz_{k_1}\\
             &=\frac{d_{k}^{k}}{k!}P\{\Gamma_l \geq d_{k+l}, 1\leq l \leq k_1 -k\}\\
		&=\frac{d_{k}^{k}}{k!}Q_k.
	\end{align*} 
\end{proof}

%\section*{Acknowledgements}
%...
%And this is an acknowledgements section with a heading that was produced by the
%$\backslash$section* command. Thank you all for helping me writing this
%\LaTeX\ sample file. See \ref{suppA} for the supplementary material example.

%\begin{supplement}
%\sname{Supplement A}\label{suppA}
%\stitle{Title of the Supplement A}
%\slink[url]{http://www.e-publications.org/ims/support/dowload/imsart-ims.zip}
%\sdescription{Dum esset rex in
%accubitu suo, nardus mea dedit odorem suavitatis. Quoniam confortavit
%seras portarum tuarum, benedixit filiis tuis in te. Qui posuit fines tuos}
%\end{supplement}

\bibliography{allMyReferences}
%\bibliography{mybib}{}
\bibliographystyle{ieeetr}

\end{document}